%% file: Control_KCS.tex
\documentclass{article}
\usepackage{amsmath,amssymb,amsthm,subfigure,float,url}
\usepackage[colorlinks=true]{hyperref}
\usepackage{fancybox,graphicx,epsf,color}

\def\R{\mathbb{R}}
\def\N{\mathbb{N}}

\newcommand{\eps}{\varepsilon}

\renewcommand{\geq}{\geqslant}
\renewcommand{\leq}{\leqslant}

\newtheorem{theorem}{Theorem}

\newtheorem{proposition}{Proposition}
\newtheorem{corollary}{Corollary}
\newtheorem{definition}{Definition}
\newtheorem{lemma}{Lemma}
\theoremstyle{definition}
\theoremstyle{definition}\newtheorem{remark}{Remark}

\newcommand{\supp}{\mathrm{supp}}
\newcommand{\dxv}{\,dx\,dv}
\renewcommand{\SS}{\mathcal{S}}
\newcommand{\SSS}{\mathbb{S}}
\newcommand{\TT}{\mathcal{T}}
\newcommand{\TTT}{\mathbb{T}}

\title{Control to flocking of the kinetic Cucker-Smale model}

\author{
Benedetto Piccoli\thanks{Department of Mathematical Sciences, Rutgers University - Camden, Camden, NJ ({\tt piccoli@camden.rutgers.edu}).}
\and
Francesco Rossi\thanks{Aix Marseille Universit\'e, CNRS, ENSAM, Universit\'e de Toulon, LSIS UMR 7296,13397, Marseille, France ({\tt francesco.rossi@lsis.org}).}
\and
Emmanuel Tr\'elat\thanks{Sorbonne Universit\'es, UPMC Univ Paris 06, CNRS UMR 7598, Laboratoire Jacques-Louis Lions, Institut Universitaire de France, F-75005, Paris, France ({\tt emmanuel.trelat@upmc.fr}).}
}

\date{}

\begin{document}

\maketitle

\begin{abstract}
The well-known Cucker-Smale model is a macroscopic system reflecting flocking, i.e. the alignment of velocities in a group of autonomous agents having mutual interactions. In the present paper, we consider the mean-field limit of that model, called \textit{the kinetic Cucker-Smale model}, which is a transport partial differential equation involving nonlocal terms. 
It is known that flocking is reached asymptotically whenever the initial conditions of the group of agents are in a favorable configuration. 
For other initial configurations, it is natural to investigate whether flocking can be enforced  by means of an appropriate external force, applied to an adequate time-varying subdomain. 

In this paper we prove that we can drive to flocking any group of agents governed by the kinetic Cucker-Smale model, by means of a \textit{sparse} centralized control strategy, and this, for any initial configuration of the crowd. Here, ``sparse control" means that the action at each time is limited over an arbitrary proportion of the crowd, or, as a variant, of the space of configurations; ``centralized'' means that the strategy is computed by an external agent knowing the configuration of all agents. We stress that we do not only design a control function (in a sampled feedback form), but also a time-varying control domain on which the action is applied. The sparsity constraint reflects the fact that one cannot act on the whole crowd at every instant of time.

Our approach is based on geometric considerations on the velocity field of the kinetic Cucker-Smale PDE, and in particular on the analysis of the particle flow generated by this vector field. The control domain and the control functions are designed to satisfy appropriate constraints, and such that, for any initial configuration, the velocity part of the support of the measure solution asymptotically shrinks to a singleton, which means flocking.
\end{abstract}

\medskip


\medskip


\section{Introduction}
In recent years, the study of collective behavior of a crowd of autonomous agents has drawn a great interest from scientific communities, e.g., in civil engineering (for evacuation problems), robotics (coordination of robots), computer science and sociology (social networks), and biology (crowds of animals). In particular, it is well known that some simple rules of interaction between agents can provide the formation of special patterns, like in formations of bird flocks, lines, etc. This phenomenon is often referred to as \textit{self-organization}.
Beyond the problem of analyzing the collective behavior of a ``closed'' sytem, it is interesting to understand what changes of behavior can be induced by an external agent (e.g., a policy maker) to the crowd. In other words, we are interested in understanding how one can act on a group of agents whose movement is governed by some continuous model of collective behavior.
For example, one can try to enforce the creation of patterns when they are not formed naturally, or break the formation of such patterns. This is the problem of \textit{control of crowds}, that we address in this article for the kinetic (PDE) version of the celebrated Cucker-Smale model introduced in \cite{CS}.

From the analysis point of view, one needs to pass from a big set of simple rules for each individual to a model capable of capture the dynamics of the whole crowd. This can be solved via the so called mean-field process, that permits to consider the limit of a set of ordinary differential equations (one for each agent) to a partial differential equation for the density of the whole crowd.

In view of controlling such models, two approaches do emerge: one can either address a control problem for a finite number of agents, solve it and then pass to the limit in some appropriate sense (see, e.g., \cite{Mattia,FPR,FS}); or one can directly address the control problem for the PDE model: this is the point of view that we adopt in this paper.

\medskip

In this paper, we consider the \textit{controlled kinetic Cucker-Smale equation}
\begin{equation}\label{kinetic_CS}
\boxed{
\partial_t \mu +\langle v, \mathrm{grad}_x \mu\rangle+\mathrm{div}_v \left( (\xi[\mu]+\chi_\omega u) \mu \right)=0 ,
}
\end{equation}
where $\mu(t)$ is a probability measure on $\R^d\times\R^d$ for every time $t$ (if $\mu(t,x,v)=f(t,x,v)\, \dxv$, then $f$ is the density of the crowd), with $d\in\N^*$ fixed, and $\xi[\mu]$ is the \textit{interaction kernel}, defined by
\begin{equation}\label{def_xi}
\xi[\mu](x,v)=\int_{\R^d\times\R^d} \phi(\Vert x-y\Vert) (w-v)\,d\mu(y,w) ,
\end{equation}
for every probability measure $\mu$ on $\R^d\times\R^d$, and for every $(x,v)\in\R^d\times\R^d$. 
The function $\phi:\R\to\R$ is a nonincreasing positive function, accounting for the influence between two particles, depending only on their mutual distance.
The term $\chi_\omega u$ is the control, which consists of:
\begin{itemize}
\item the control set $\omega=\omega(t)\subset\R^d\times\R^d$ (on which the control force acts),
\item the control force $u=u(t,x,v)\in\R^d$.
\end{itemize}
We stress that the control is not only the force $u$, but also the set $\omega$ on which the force acts.
Physically, $u$ represents an acceleration (as in \cite{CS-control} for the finite-dimensional model), and $\omega(t)$ is the portion of the space-velocity space on which one is allowed to act at time $t$. It is interesting to note that, in the usual literature on control, it is not common to consider a subset of the space as a control. 

There are many results in the literature treating the problem of self-organization of a given crowd of agents, like flocks of birds (see \cite{BCCCCGLOPPVZ09, cavagna,CKFL05, mecholsky,PE99, vicsek,YEECBKMS09}), pedestrian crowds (see \cite{crpito11,realistic-crowd}), robot formations (see \cite{MR2000132,Kumar}), or socio-economic networks (see \cite{beheto12,Krause-Hegselmann}). A nonexhaustive list of references on the subject from the scientific, biological, and even politic points of view are the books \cite{axelrod,camazine,helbingbook,jackson} and the surveys \cite{castellano,helbingreview,MotschTadmor_SIREV,olfati,reynolds,vicsek}. In particular, in \cite{olfati,reynolds} the authors classify interaction forces into flocking centering, collision avoidance and velocity matching. Clearly, both the Cucker-Smale and the kinetic Cucker-Smale models deal with velocity matching forces only. 

A fundamental tool for this topic is the notion of mean-field limit, where one obtains a distribution of crowd by considering a crowd with a finite number $N$ and by letting $N$ tend to the infinity. The result of the mean-field limit is also called a ``kinetic'' model. For this reason, we call the model in \eqref{kinetic_CS} the \textit{kinetic Cucker-Smale model}. The mean-field limit of the finite-dimensional Cucker-Smale model was first derived in \cite{hatad} (see also \cite{carfor,haliu}). Other mean-field limits of alignment models are studied in \cite{carrillo21,degond,tadmortan}. Many other mean-field limits of models defined for a finite number of agents have been studied (see, e.g., \cite{canuto,duering,toscani}).

Assuming now that one is allowed to apply an action on the system, it is very natural to try to steer the system asymptotically to flocking. This may have many applications. We refer the reader to examples of centralized and distributed control algorithms in \cite{bullo} (see also the references therein). All these examples are defined for a finite number of agents, possibly very large. Instead, the control of mean-field transport equations is a recent field of research (see, e.g., \cite{FPR,Leautaud}, see also stochastic models in \cite{markov}).

\medskip

Note that \eqref{kinetic_CS} is a transport PDE with \textit{nonlocal interaction terms}. As it is evident from the expression of $\xi[\mu]$, the velocity field $\xi[\mu]$ acting on the $v$ variable depends globally on the measure $\mu$. In other words, if $\mu$ has a density $f$, then $\xi[\mu](x)$ is not uniquely determined by the value of $f(x)$, but it depends on the value of $f$ in the whole space $\R^d\times\R^d$. 
Existence, uniqueness and regularity of solutions for this kind of equation with no control term ($u=0$) have been established quite recently (see \cite{ambrosio}). We will establish the well-posedness of \eqref{kinetic_CS} in Section \ref{s-pde}.

In the present paper, our objective is to design an explicit control $\chi_\omega u$, satisfying realistic constraints, able to steer the system \eqref{kinetic_CS} from any initial condition to flocking. Let us first recall what is flocking.

\medskip

Throughout the paper, we denote by $\mathcal{P}(\R^d\times\R^d)$ the set of probability measures on $\R^d\times\R^d$, by $\mathcal{P}_c(\R^d\times\R^d)$ the set of probability measures on $\R^d\times\R^d$ with compact support, and by $\mathcal{P}^{ac}_c(\R^d\times\R^d)$ the set of probability measures on $\R^d\times\R^d$ with compact support and that are absolutely continuous with respect to the Lebesgue measure. We denote with $\supp(\mu)$ the support of $\mu$.

Given a solution $\mu\in C^0(\R,\mathcal{P}_c(\R^d\times\R^d))$ of \eqref{kinetic_CS}, we define the space barycenter $\bar x(t)$ and the velocity barycenter $\bar v(t)$ of $\mu(t)$ by
\begin{equation}\label{def_space_velocity_barycenter}
{\bar x}(t)=\int_{\R^d\times\R^d} x \, d\mu(t)(x,v),\qquad
{\bar v}(t)=\int_{\R^d\times\R^d} v \, d\mu(t)(x,v) ,
\end{equation}
for every $t\in\R$.
If there is no control ($u=0$), then $\bar v(t)$ is constant in time. If there is a control, then, as we will see further, we have $\dot {\bar x}(t)={\bar v}(t)$ and $\dot {\bar v}(t)=\int_{\omega(t)} u(t,x,v)\, d\mu(t)(x,v)$.

\begin{definition} \label{def-flock}
Let $\mu\in C^0(\R,\mathcal{P}_c(\R^d\times\R^d))$ be a solution of \eqref{kinetic_CS} with $u\equiv 0$. We say that $\mu$ converges to flocking if the two following conditions hold:
\begin{itemize}
\item there exists $X_M>0$ such that $\supp(\mu(t))\subseteq B({\bar x}(t),X_M)\times \R^d$ for every $t>0$;
\item $\Lambda(t)=\int_{\R^d\times\R^d} |v-{\bar v}|^2\,d\mu(t)\longrightarrow 0$ as $t\rightarrow +\infty$.
\end{itemize}
We also define the {\em flocking region} as the set of configurations $\mu^0\in \mathcal{P}_c(\R^d\times\R^d)$ such that the solution of \eqref{kinetic_CS} with $u\equiv 0$ and initial data $\mu(0)=\mu^0$ converges to flocking.
\end{definition}

Note that, defining the velocity marginal of $\mu(t)$ by $\mu_v(t)(A)=\mu(t)(\R^d\times A)$ for every measurable subset $A$ of $\R^d$, this definition of flocking means that $\mu_v(t)$ converges (vaguely) to the Dirac measure $\delta_{\bar v}$, while the space support remains bounded around $\bar x(t)$.

Intuitively, $\mu(t)$ is the distribution at time $t$ of a given crowd of agents in space $x$ and velocity $v$. Asymptotic flocking means that, in infinite time, all agents tend to align their velocity component, as a flock of birds that, asymptotically, align all their velocities and then fly in a common direction. Flocking can also be more abstract and the variable $v$ can represent, for instance, an opinion: in that case flocking means consensus. Then, the techniques presented here may be adapted for similar problems for consensus (reaching a common value for all state variables) or alignment (reaching a common value in some coordinates of the state variable).

In order to steer a given crowd to flocking, the control term in \eqref{kinetic_CS} means that we are allowed to act with an external force, of amplitude $u(t,x,v)$, supported on the control domain $\omega(t)$. Our objective is then to design appropriate functions $t\mapsto u(t,\cdot)$ and $t\mapsto\omega(t)$ leading to flocking.
In order to reflect the fact that, at every instant of time, one can act only on a small proportion of the crowd, with a force of finite amplitude, we impose some constraints on the control function $u$ and on the control domain $\omega$.

Let $c>0$ be arbitrary.
We consider the class of controls $\chi_\omega u$, where $u\in L^\infty(\R\times\R^d\times\R^d)$ and $\omega(t)$ is a measurable subset of $\R^d\times\R^d$ for every time $t$, satisfying the constraints:
\begin{equation}\label{cont_u}
\Vert u(t,\cdot,\cdot)\Vert _{L^\infty(\R^d\times\R^d)}\leq 1 ,
\end{equation}
for almost every time $t$ and 
\begin{equation}\label{cont_omega}
\mu(t)(\omega(t)) = \int_{\omega(t)} d\mu(t)(x,v) \leq c  ,  
\end{equation}
for every time $t$. 

The constraint \eqref{cont_u} means that the control function (representing the external action) is bounded, and the constraint \eqref{cont_omega} means that one is allowed to act only on \textit{a given proportion $c$ of the crowd}.
In \eqref{cont_omega}, $\mu(t)$ is the solution at time $t$ of \eqref{kinetic_CS}, associated with the control $\chi_\omega$. The existence and uniqueness of solutions will be proved in the following.

As a variant of \eqref{cont_omega}, we will consider the following constraint as well:
\begin{equation}\label{cont_omega_variante}
\vert\omega(t)\vert=\int_{\omega(t)} \dxv\leq c ,
\end{equation}
for every time $t$. 

The fact that the action is limited either to a given (possibly small) proportion of the crowd, or of the space of configurations, is related to the concept of \textit{sparsity}, in which one aims at controlling a system (or, reconstructing some information) with a minimal amount of action, like a shepherd dog trying to maintain a flock of sheeps.

\medskip

Note that it is obviously necessary to allow the control domain to move because, if the control domain $\omega$ is fixed (in time), then it is not difficult to construct initial data $\mu^0$ that cannot be steered to flocking, for any control function $u$. Indeed, consider the example of a particle model without control that is not steered to flocking\footnote{An example in dimension one with two agents for the finite-dimensional Cucker-Smale model is given in \cite{CS}.} and consider a fixed control set $\omega$, disjoint of the trajectories of the system (for example a control set with velocity coordinates that are larger than the maximum of the velocities of the particles). Then, replace the particles with absolutely continuous measures centered around them, that is, $(x,v)$ is replaced with $\chi_{[x-\eps,x+\eps]\times[v-\eps,v+\eps]}$. Choosing $\eps$ sufficiently small, the dynamics of the resulting measure with the same $\omega$ is close to the dynamics of the particle model, hence it does not converge to flocking.

\medskip

In this paper, we will prove the following result.

\begin{theorem}\label{mainthm}
Let $c>0$ be arbitrary.
For every $\mu^0\in\mathcal{P}_c^{ac}(\R^d\times\R^d)$, there exists a control $\chi_\omega u$, satisfying the constraints \eqref{cont_u} and \eqref{cont_omega} (or, as a variant, the constraints \eqref{cont_u} and \eqref{cont_omega_variante}), and there exists a unique solution $\mu\in C^0(\R; \mathcal{P}_c^{ac}(\R^d\times\R^d))$ of \eqref{kinetic_CS} such that $\mu(0)=\mu^0$, and converging to flocking as $t$ tends to $+\infty$.
\end{theorem}

Note that, given any initial measure that is absolutely continuous and of compact support, the control $\chi_\omega u$ that we design generates a solution of \eqref{kinetic_CS} that remains absolutely continuous and of compact support. It is important to note that, from a technical point of view, we will be able to prove existence and uniqueness of the solution as long as the control function $\chi_\omega u$ remains Lipschitz with respect to state variables. Since $\mu$ converges to flocking, $\mu$ becomes singular only in infinite time.

\begin{remark} In Section \ref{sec_proof_mainthm}, we will design an {\em explicit} control $\chi_\omega u$ steering the system \eqref{kinetic_CS} to flocking, with the following properties:
\begin{itemize}
\item $\omega(t)$ is piecewise constant in $t$;
\item $u(t,x,v)$ is piecewise constant in $t$ for $(x,v)$ fixed, continuous and piecewise linear in $(x,v)$ for $t$ fixed;
\item for any initial configuration $\mu^0\in\mathcal{P}_c^{ac}(\R^d\times\R^d)$, there exists a time $T(\mu^0)\geq 0$ such that $u(t,x,v)=0$ for every $t>T(\mu^0)$.
\end{itemize}
Note that the control that we design is ``centralized'', in the sense that the external agent acting on the crowd has to know the configuration of all agents, at every instant of time.

As we will see, the solution $\mu(t)$ of \eqref{kinetic_CS} is exactly the pushforward of the initial measure under the controlled particle flow, which is the flow of a given vector field involving the control term. Our strategy for designing a control steering the system to flocking consists in interpreting it as a particle system and in choosing the control domain and the control function such that the velocity field points inwards the domain, so that the size of the velocity support of $\mu(t)$ decreases (exponentially) in time. Our construction goes by considering successive (small enough) intervals of times along which the control domain remains constant, whence the property of being piecewise constant in time.

The third item above means that the control is not active for every time $t>0$. Indeed, we prove in Theorem \ref{c-flocksize} that, for the uncontrolled equation \eqref{kinetic_CS} (i.e., with $u\equiv 0$), if the support of $\mu(t_0)$ is ``small enough" at some time $t_0$ then $\mu$ converges to flocking, without requiring any action on the crowd. As a consequence, if the initial crowd is in a favorable configuration at the initial time (if it is not too much dispersed), then the crowd will naturally converge to flocking, without any control.
Then our control strategy consists of applying an appropriate control, until $\mu(t)$ reaches the \textit{flocking region} defined in Definition \ref{def-flock}, in which its support is small enough so that $\mu$ converges naturally (without any control) to flocking. This means that we switch off the control after a time $T(\mu^0)$, depending on the initial distribution $\mu^0$: it is expected that $T(\mu^0)$ is larger as the initial measure $\mu^0$ is more dispersed.

We stress that, in our main result (Theorem \ref{mainthm}), we do not only prove the existence of a control driving any initial crowd to flocking. Our procedure is constructive.
In our strategy, we construct a control action $u$ depending on $(t,x,v)$, and we design a control domain $\omega$ depending on $\mu(t)$. Hence, in this sense, we design a \textit{sampled feedback}. The control domain is piecewise constant in time, but this piecewise constant domain is designed in function of $\mu$.
\end{remark}

\begin{remark}\label{rem_intro1} 
In \cite{CS-control,CS-control-1}, the concept of componentwise sparse control was introduced, meaning that, for a crowd of $N$ agents whose dynamics are governed by the finite-dimensional Cucker-Smale system, one can act, at every instant of time, only on one agent.
At this step an obvious remark has to be done. In finite dimension, it is intuitive that the action on only one agent can have some consequences for the whole crowd, because of the (even weak) mutual interactions. In infinite dimension, this property is necessarily lost and should be replaced by the action on a small proportion of the population. More precisely, assume that, for the finite-dimensional model, one is allowed to act on a given proportion of the total number of agents. Then, when the number of agents tends to the infinity, we expect to recover the infinite-dimensional model \eqref{kinetic_CS}, and then this means that, indeed, one can act on a given proportion of the crowd. This is a suitable and realistic way to to pass to the limit this kind of sparsity constraint. We will give in Section \ref{sec_relation_finite_dim} a precise relationship between the finite-dimensional and the infinite-dimensional models.

By the way, note that Theorem \ref{mainthm} with the control constraints \eqref{cont_u} and \eqref{cont_omega_variante} can be compared with the results of \cite{CS-control,CS-control-1}, in which sparse feedback controls were designed for the finite-dimensional Cucker-Smale model, by driving, at every instant of time, the farthest agent to the center. 
In contrast, dealing with the constraint \eqref{cont_omega} is more difficult and requires a more complicated construction.

In \cite{FS} the authors introduce another kind of feedback control. They consider a system of particles with a feedback function action over the whole domain, which is Lipschitzian and has a bounded derivative. Then they pass to the limit on the number of particles.
In contrast, in our paper the action is limited over a (moving) subdomain $\omega$, and our control $\chi_\omega u$ consists in particular of a characteristic function.
\end{remark}

The structure of the paper is the following. 

In Section \ref{s-pde}, we recall or extend some results stating the well-posedness of the kinetic Cucker-Smale equation \eqref{kinetic_CS}, and in particular we recall that a solution of \eqref{kinetic_CS} is the image measure of the initial measure through the particle flow, which is the flow associated with the time-dependent velocity field $\xi[\mu]+\chi_\omega u$ (sections \ref{sec_2.1} and \ref{sec_appli_CS}). We also provide (in Section \ref{sec_relation_finite_dim}) a precise relationship with the finite-dimensional Cucker-Smale model, in terms of the controlled particle flow.

In Section \ref{sec_CV_without_control}, we study the kinetic Cucker-Smale equation \eqref{kinetic_CS} without control (i.e., $u\equiv 0$). We provide a simple sufficient condition on the initial measure ensuring convergence to flocking, which is a slight extension of known results.

Theorem \ref{mainthm} is proved in Section \ref{sec_proof_mainthm}. In that section, after having established preliminary estimates (in Sections \ref{sec_4.1} and \ref{sec_4.2}), we first prove Theorem \ref{mainthm} in the one-dimensional case, that is, for $d=1$, in Section \ref{s-Cpop1}. Our strategy is based on geometric considerations, by choosing an adequate control, piecewise constant in time, such that the velocity field is pointing inwards the support, in such a way that the velocity support decreases in time. We apply this strategy iteratively, until we reach (in finite time) the flocking region, and then we switch off the control and let the solution evolve naturally to flocking. The general case $d>1$ is studied in Section \ref{s-SSd}. The variant, with the control constraints \eqref{cont_u} and \eqref{cont_omega_variante}, is studied in Section \ref{s-space}.

\section{Existence and uniqueness}\label{s-pde}
In this section, we provide existence and uniqueness results for \eqref{kinetic_CS}.
Note that, since $\langle v,\mathrm{grad}_x \mu\rangle = \mathrm{div}_x ( v\mu)$, the PDE \eqref{kinetic_CS} can be written as
$$
\partial_t \mu + \mathrm{div}_{(x,v)} \left( \begin{pmatrix} v \\Ê\xi[\mu]Ê+ \chi_\omega u \end{pmatrix} \mu \right) = 0 .
$$
In this form, this is a transport equation with a divergence term. Let us then recall some facts on such equations.

\subsection{Transport partial differential equations with nonlocal velocities}\label{sec_2.1}
In this section, we consider the general nonlocal transport partial differential equation
\begin{equation}\label{kinetic_CSbase}
\partial_t \mu + \mathrm{div} ( V[\mu]\mu )=0 ,
\end{equation}
where $\mu\in \mathcal{P}(\R^n)$ is a probability measure on $\R^n$, with $n\in\N^*$ fixed. 
The term $V[\mu]$ is called the \textit{velocity field} and is a nonlocal term.
Since the value of a measure at a single point is not well defined, it is important to observe that $V[\mu]$ is not a function depending on the value of $\mu$ in a given point, as it is often the case in the setting of hyperbolic equations in which $V[\mu](x)=V(\mu(x))$. Instead, one has to consider $V$ as an operator taking an as input the whole measure $\mu$ and giving as an output a global vector field $V[\mu]$ on the whole space $\R^n$. These operators are often called ``nonlocal'', as they consider the density not only in a given point, but in a whole neighborhood.

We first recall two useful definitions to deal with measures and solutions of \eqref{kinetic_CSbase}, namely the Wasserstein distance and the pushforward of measures. For more details see, e.g., \cite{villani}.
\begin{definition}
Given two probability measures $\mu$ and $\nu$ on $\R^n$, the $1$-Wasserstein distance between $\mu$ and $\nu$ is
$$
W_1(\mu,\nu) = \sup\left\{ \int_{\R^n} f \,d(\mu-\nu) \mid  f\in C^{\infty}(\R^n),\ \Vert f\Vert _{Lip}\leq 1 \right\},
$$
where $\Vert f\Vert _{Lip}$ is the Lipschitz constant\footnote{We have $\Vert f\Vert _{Lip} = \sup \left\{ \frac{\vert f(x)-f(y)\vert}{\Vert x-y\Vert} \mid x,y\in\R^n, x\neq y\right\}$.}
of the function $f$.
\end{definition}
This formula for the Wasserstein distance, which can be taken as a definition, comes from the Kantorovich-Rubinstein theorem.
Note that the topology induced by $W_1$ on $\mathcal{P}_c(\R^n)$ coincides\footnote{Actually, the distance $W_1$ metrizes the weak convergence of measures only if their first moment is finite, which is true for measures with compact support.
} with the weak topology (see \cite[Theorem 7.12]{villani}).
We now define the pushforward of measures.

\begin{definition} 
Given a Borel map $\gamma:\R^n\rightarrow\R^n$, the pushforward of a measure $\mu\in \mathcal{P}(\R^n)$ is defined by
\begin{equation*}
\gamma\#\mu(A)=\mu(\gamma^{-1}(A)),
\end{equation*}
for every measurable subset $A$ of $\R^n$.
\end{definition}

We now provide an existence and uniqueness result for \eqref{kinetic_CSbase}.

\begin{theorem}\label{t-esistenza}
We assume that, for every $\mu\in \mathcal{P}_c(\R^n)$, the velocity field $V[\mu]$ is a function of $(t,x)$ with the regularity
\begin{equation*}
V[\mu] \in L^\infty \left( \R; C^{1}(\R^n)\cap L^\infty(\R^n)  \right) ,
\end{equation*}
satisfying the following assumptions:
\begin{itemize}
\item 
there exist functions $L(\cdot)$ and $M(\cdot)$ in $L^\infty_{loc}(\R)$ such that 
$$
\left\Vert V[\mu](t,x)-V[\mu](t,y) \right\Vert \leq L(t) \Vert x-y\Vert,\qquad \Vert V[\mu](t,x)\Vert \leq M(t) (1+\Vert x\Vert ),
$$
for every $\mu\in\mathcal{P}_c(\R^n)$, every $t\in\R$ and all $(x,y)\in\times\R^n$;
\item 
there exists a function $K(\cdot)$ in $L^\infty_{loc}(\R)$ such that  
$$
\left\Vert V[\mu]-V[\nu] \right\Vert _{L^\infty(\R;C^0(\R^n))} \leq K(t) W_1(\mu,\nu),
$$
for all $(\mu,\nu)\in(\mathcal{P}_c(\R^n))^2$.
\end{itemize}
Then, for every $\mu^0\in\mathcal{P}_c(\R^n)$, the Cauchy problem
\begin{equation}\label{e-cauchy}
\partial_t\mu+\mathrm{div}(V[\mu] \mu)=0, \quad \mu_{|_{t=0}}=\mu^0,
\end{equation}
has a unique solution in $C^0(\R;\mathcal{P}_c(\R^n))$, where $\mathcal{P}_c(\R^n)$ is endowed with the weak topology, and $\mu$ is locally Lipschitz with respect to $t$, in the sense of the Wasserstein distance $W_1$. Moreover, if $\mu^0\in\mathcal{P}^{ac}_c(\R^n)$, then $\mu(t)\in\mathcal{P}^{ac}_c(\R^n)$, for every $t\in\R$.

Furthermore, for every $T>0$, there exists $C_T>0$ such that 
\begin{equation}\label{e-stabilitypde}
W_1(\mu(t),\nu(t))\leq e^{C_T t}W_1(\mu(0),\nu(0)),
\end{equation}
for all solutions $\mu$ and $\nu$ of \eqref{e-cauchy} in $C^0([0,T];\mathcal{P}_c(\R^n))$.

Moreover, the solution $\mu$ of the Cauchy problem \eqref{e-cauchy} can be made explicit as follows.
Let $\Phi(t)$ be the flow of diffeomorphims of $\R^n$ generated by the time-dependent vector field $V[\mu]$, defined as the unique solution of the Cauchy problem $\dot\Phi(t) = V[\mu(t)] \circ \Phi(t)$, $\Phi(0)=\mathrm{Id}_{\R^n}$, or in other words,
$$
\partial_t\Phi(t,x) = V[\mu(t)](t,\Phi(t,x)),\quad \Phi(0,x)=x.
$$
Then, we have
$$
\mu(t) = \Phi(t)\#\mu(0),
$$
that is, $\mu(t)$ is the pushforward of $\mu^0$ under $\Phi(t)$.
\end{theorem}

\begin{proof}
The proof is a slight generalization of results established in \cite{pedestrian}. We have two differences: the first is the fact that $V[\mu]$ is not uniformly bounded globally in space (i.e., $\Vert V[\mu](x)\Vert \leq M$ for every $\mu\in\mathcal{P}_c(\R^d))$ and every $x\in\R^n$), but only with sublinear growth; the second is the fact that the velocity field is time-dependent.

Let $T\in \R$ be arbitrary. Using a sample-and-hold method, and following \cite{pedestrian}, it is possible to build a sequence of approximated solutions $\mu^k\in C^0([0,T],\mathcal{P}_c(\R^n))$, which converges to a solution $\mu^*\in C^0([0,T],\mathcal{P}_c(\R^n))$ of \eqref{e-cauchy}. Then, a simple Gronwall estimate yields \eqref{e-stabilitypde}, that in turn implies uniqueness of the solution of \eqref{e-cauchy}. Since $T$ is arbitrary, the global result follows.
The last part is established in a standard way (see, e.g., \cite[Theorem 5.34]{villani}), even though the velocity field is time-dependent.
\end{proof}

\begin{remark}
Theorem \ref{t-esistenza} can be generalized to mass-varying transport PDE's, that is, in presence of sources (see \cite{genwass}).
\end{remark}

\subsection{Application to the kinetic Cucker-Smale equation}\label{sec_appli_CS}
In the case of the kinetic Cucker-Smale equation \eqref{kinetic_CS}, we have $n=2d$, and for a given control $\chi_\omega u$ the time-dependent velocity field is given by 
$$
V_{\omega,u}[\mu(t)](t,x,v) = \begin{pmatrix} v \\ \xi[\mu](x,v) + \chi_{\omega(t)} u(t,x,v) \end{pmatrix} .
$$
We denote by $\Phi_{\omega,u}(t)$ the so-called ``controlled particle flow", generated by the time-dependent vector field $V_{\omega,u}[\mu(t)]$, defined by $\partial_t\Phi_{\omega,u}(t,x) = V_{\omega,u}[\mu(t)](t,\Phi_{\omega,u}(t,x))$ and $\Phi_{\omega,u}(0,x)=x$.
The flow $\Phi_{\omega,u}(t)$ is built by integrating the \textit{characteristics}
\begin{equation}\label{controlledparticleflow}
\dot x(t) = v(t),\qquad \dot v(t) = \xi[\mu(t)](x(t),v(t)) + \chi_{\omega(t)} u(t,x(t),v(t)),
\end{equation}
which give the evolution of (controlled) particles: the trajectory $t\mapsto (x(t),v(t))$ is called the particle trajectory passing through $(x(0),v(0))$ at time $0$, associated with the control $\chi_\omega u$.
From Theorem \ref{t-esistenza}, we have the following result.

\begin{corollary}\label{cor_appli_CS_control}
Let $u\in L^\infty(\R\times\R^d\times\R^d,\R^d)$ be a control function, and, for every time $t$, let $\omega(t)$ be a Lebesgue measurable subset of $\R^d\times\R^d$. Let $\mu^0\in\mathcal{P}_c(\R^d\times\R^d)$.
The controlled kinetic Cucker-Smale equation \eqref{kinetic_CS} has a unique solution $\mu\in C^0(\R,\mathcal{P}(\R^d\times\R^d))$ such that $\mu(0)=\mu^0$, and moreover we have
$$
\mu(t)=\Phi_{\omega,u}(t)\#\mu^0,
$$
for every $t\in\R$. Moreover, if $\mu^0\in\mathcal{P}_c^{ac}(\R^d\times\R^d)$ then $\mu(t)\in\mathcal{P}_c^{ac}(\R^d\times\R^d)$ for every $t\in\R$, and
\begin{equation}\label{controlledpropag_supp}
\supp(\mu(t)) = \Phi_{\omega,u}(t) ( \supp(\mu^0) ).
\end{equation}
\end{corollary}

\begin{remark} \label{r-ck}
If the initial measure $\mu(0)$ has a density with respect to the Lebesgue measure that is a function of class $C^k$ on $\R^d\times\R^d$, and if the vector field is also of class $C^k$,  then, clearly, we have $\mu(t)=f(t)\dxv$ with $f$ of class $C^k$ as well, because of the property of pushforward of measures.

In this paper, we do not investigate further the $C^k$ regularity from the control point of view: our control function $u$ will be designed in a Lipschitz way with respect to the space-velocity variables. Nevertheless, we could easily modify the definition of $u$ outside of the sets where $u=0$ and $u=1$, in order to design $u$ as a function of class $C^k$ that drives the solution to flocking, and that also keeps $C^k$ regularity if the initial data is of class $C^k$ (see also Remark \ref{r-ck2} further).
\end{remark}

\subsection{Relationship with the finite-dimensional Cucker-Smale model}\label{sec_relation_finite_dim}
In this section, we explain in which sense the kinetic equation \eqref{kinetic_CS} is the natural limit, as the number of agents tends to infinity, of the classical finite-dimensional Cucker-Smale model (whose controlled version is considered in \cite{CS-control,CS-control-1}), and we explain the natural relationship between them in terms of particle flow.

\subsubsection{The finite-dimensional Cucker-Smale model}
Consider $N$ agents evolving in $\R^d$, and interacting together. 
We denote with $(x_i,v_i)$ the space-velocity coordinates of each agent, for $i=1,\ldots,N$. The general Cucker-Smale model (without control) is written as
\begin{equation}\label{efdCS}
\begin{split}
\dot x_i(t) &= v_i(t), \\
\dot v_i(t) &=  \frac{1}{N} \sum_{j=1}^N \phi(\Vert x_j(t)-x_i(t)\Vert) ( v_j(t) - v_i(t)), \quad i=1,\dots N,
\end{split}
\end{equation}
where $\phi:\R\to\R$ is a nonincreasing positive function, modelling the influence between two individuals (which depends only on their mutual distance). 
This simple model, initially introduced in \cite{CS}, has many interesting features. The most interesting property is that the model reflects the ability of the crowd to go to self-organization for favorable initial configurations. Indeed, if the influence of each agent on the others is sufficiently large (that is, if $\phi$ does not decrease too fast), then the crowd converges to flocking, in the sense that all variables $v_i(t)$ converge to the common mean velocity ${\bar v}$. By analogy with birds flocks, this phenomenon was called \textit{flocking} (see \cite{CS}). 

To be more precise, first observe that the velocity barycenter ${\bar v}=\frac{1}{N}\sum_{i=1}^N v_i(t)$ is constant in time, and that, defining the space barycenter ${\bar x}(t)=\frac{1}{N}\sum_{i=1}^N x_i(t)$, we have $\dot {\bar x}(t)={\bar v}$. Then, define $\Gamma(t)=\sum_{i=1}^N |x_i(t)-{\bar x}(t)|^2$ and $\Lambda(t)=\sum_{i=1}^N |v_i(t)-{\bar v}|^2$. It is proved in \cite{HaHaKim,haliu} that, is $\Lambda(0)<\int_{\Gamma(0)}^\infty \phi(x)\,dx$, then $\Lambda(t)\rightarrow 0$ as $t\rightarrow+\infty$, that is, the crowd converges to flocking.
At the contrary, if the initial configuration is ``too dispersed'' and/or the interaction between agents is ``too weak'', then the crowd does not converge to flocking (see \cite{CS}).

Many variants and generalizations were proposed in the recent literature, but it is not our objective, here, to list them.
A controlled version of \eqref{efdCS} was introduced and studied in \cite{CS-control,CS-control-1}, consisting of adding controls at the right-hand side of the equations in $v_i$, turning the system into
\begin{equation}\label{efdCS_u}
\begin{split}
\dot x_i(t) &= v_i(t), \\
\dot v_i(t) &=  \frac{1}{N} \sum_{j=1}^N \phi(\Vert x_j(t)-x_i(t)\Vert) ( v_j(t) - v_i(t)) + u_i(t), \quad i=1,\dots N,
\end{split}
\end{equation}
where the controls $u_i$, taking their values in $\R^d$, can be constrained in different ways.
Since it is desirable to control the system \eqref{efdCS_u} with a minimal number of actions (for instance, acting on few agents only), in \cite{CS-control,CS-control-1} the concept of \textit{sparse control} was introduced. This means that, at every instant of time at most one component of the control is active, that is, for every time $t$ all $u_i(t)$ but one are zero.\footnote{This property was called \textit{componentwise sparsity}. Actually, in order to prevent the system from chattering in time, also a notion of \textit{time sparsity} was considered in \cite{CS-control,CS-control-1}.}
It was shown how to design a sparse feedback control $(t,x,v)\mapsto u(t,x,v)$ steering the system \eqref{efdCS_u} asymptotically to flocking.

\subsubsection{Towards the kinetic Cucker-Smale model}
In the absence of control, the finite-dimensional Cucker-Smale model \eqref{efdCS} was generalized to an infinite-dimensional setting in measure spaces via a mean-field limit process in \cite{carfor,haliu,hatad}. The limit is taken by letting the number of agents $N$ tend to the infinity. Considering the pointwise agents as Dirac masses, it is easy to embed the dynamics \eqref{efdCS_u} in the space of measures, and using Corollary \eqref{cor_appli_CS_control}, we infer the following result.

\begin{proposition}\label{prop1}
Let $u\in L^\infty(\R\times\R^d\times\R^d,\R^d)$ be a control function, and, for every time $t$, let $\omega(t)$ be a Lebesgue measurable subset of $\R^d\times\R^d$. Let $\mu^0\in\mathcal{P}(\R^d\times\R^d)$ be defined by $\mu^0=\frac{1}{N}\sum_{i=1}^N\delta_{(x_i^0,v_i^0)}$, for some $(x_i^0,v_i^0)\in\R^d\times\R^d$, $i=1,\ldots,N$.
Then the unique solution of \eqref{kinetic_CS} such that $\mu(0)=\mu^0$, corresponding to the control $\chi_\omega u$, is given by
$$
\mu(t)=\frac{1}{N}\sum_{i=1}^N\delta_{(x_i(t),v_i(t))},
$$
where $(x_i(t),v_i(t))$, $i=1,\ldots,N$, are solutions of
\begin{equation*}
\begin{split}
\dot x_i(t) &= v_i(t), \\
\dot v_i(t) &=  \frac{1}{N} \sum_{j=1}^N \phi(\Vert x_j(t)-x_i(t)\Vert) ( v_j(t) - v_i(t)) +  \chi_{\omega(t)}(x_i(t),v_i(t)) u(t,x_i(t),v_i(t)), 
\end{split}
\end{equation*}
such that $x_i(0)=x_i^0$ and $v_i(0)=v_i^0$, for $i=1,\ldots, N$.
\end{proposition}

\begin{proof}
The equation \eqref{kinetic_CS} being stated in the sense of measures, we have, for any $g\in C^\infty(\R^d\times\R^d)$,
\begin{equation*}
\begin{split}
0 = &\ \partial_t \int g(x,v)\, d \mu(t)(x,v) + \int g(x,v) \, \mathrm{div}_x (v\mu(t)(x,v)) \\
& \qquad\qquad + \int g(x,v)\, \mathrm{div}_v \left( ( \xi[\mu(t)](x,v)  + \chi_{\omega(t)}(x,v) u(t,x,v) )\mu(t)(x,v) \right) \\
= & \ 
\partial_t \int g(x,v)\, d \mu(t)(x,v) - \int  \langle v,\mathrm{grad}_x g(x,v) \rangle \, d\mu(t)(x,v) \rangle \\
& \qquad\qquad - \int \left\langle  \xi[\mu(t)](x,v)  + \chi_{\omega(t)}(x,v) u(t,x,v)) , \mathrm{grad}_v g(x,v) \right\rangle  d\mu(t)(x,v) ,
\end{split}
\end{equation*}
and taking $\mu(t)=\frac{1}{N}\sum_{i=1}^N\delta_{(x_i(t),v_i(t))}$ gives
\begin{equation*}
\begin{split}
& \frac{1}{N} \sum_{i=1}^N \left( \langle \dot x_i(t),\mathrm{grad}_x g(x_i(t),v_i(t)) \rangle + \langle \dot v_i(t),\mathrm{grad}_v g(x_i(t),v_i(t))\rangle \right) \\
=&\ \frac{1}{N} \sum_{i=1}^N \Big( \langle v_i(t),\mathrm{grad}_x g(x_i(t),v_i(t)) \rangle \\
& \qquad\qquad+ \left\langle  \xi[\mu](x_i(t),v_i(t))  + \chi_\omega(x_i(t),v_i(t)) u(t,x_i(t),v_i(t)) , \mathrm{grad}_v g(x_i(t),v_i(t)) \right\rangle \Big) ,
\end{split}
\end{equation*}
with
$$
\xi[\mu(t)](x,v) = \frac{1}{N} \sum_{j=1}^N \phi(\Vert x_j(t)-x\Vert) ( v_j(t) - v) ,
$$
from which we infer the finite-dimensional Cucker-Smale system stated in the proposition (it suffices to consider functions $g$ localized around any given particle $(x_i(t),v_i(t))$). We conclude by uniqueness, using Corollary \eqref{cor_appli_CS_control}.
\end{proof}

\begin{remark}
In accordance with the discussion done in Remark \ref{rem_intro1} concerning sparsity, we see clearly that the control domain $\omega(t)$, in finite dimension, represents the agents on which one can act at the instant of time $t$.
This shows that the way to pass to the limit a sparsity control constraint on the finite-dimensional model is to consider proportions either of the total crowd or of the space of configurations.
\end{remark}

\section{Convergence to flocking without control}\label{sec_CV_without_control}
In this section, we investigate the kinetic Cucker-Smale equation \eqref{kinetic_CS} without control, that is, we assume that $u\equiv 0$.

First of all, note that, as in finite dimension, the velocity barycenter ${\bar v}=\int_{\R^d\times\R^d} v\,d\mu(t)$ is constant in time, and the space barycenter ${\bar x}(t)=\int_{\R^d\times\R^d} x\,d\mu(t)$ is such that $\dot {\bar x}(t)={\bar v}$ (see, e.g., \cite[Prop. 3.1]{hatad}).

In the following theorem, we provide a simple sufficient condition on the initial probability measure ensuring flocking, in the spirit of results established in \cite{carfor,haliu}. 

\begin{theorem}\label{c-flocksize}
Let $\mu^0\in \mathcal{P}_c(\R^d\times\R^d)$. We set ${\bar x}^0=\int_{\R^d} x\,d\mu^0(x,v)$ and ${\bar v}=\int_{\R^d} v\,d\mu^0(x,v)$ (space and velocity barycenters of $\mu^0$), and we define the space and velocity support sizes
\begin{equation*}
\begin{split}
X^0&=\inf\left\{X\geq 0\mid \supp(\mu^0)\subset  B({\bar x}^0,X)\times \R^d\right\},\\
V^0&=\inf\left\{V\geq 0\mid \supp(\mu^0)\subset  \R^d\times B({\bar v},V)\right\}.
\end{split}
\end{equation*}
Let $\mu$ be the unique solution of \eqref{kinetic_CS} (with $u\equiv 0$) such that $\mu(0)=\mu^0$. If 
\begin{equation}\label{e-flock}
V^0< \int_{X^0}^{+\infty} \phi(2 x)\,dx,
\end{equation}
then there exists $X_M>0$ such that
\begin{equation}\label{e-flockcont}
\supp(\mu(t))\subset  B({\bar x}^0+t {\bar v},X_M)\times B\left({\bar v},V^0 e^{-\phi(2X_M)t}\right),
\end{equation}
for every $t\geq 0$.
In particular, $\mu(t)$ converges to flocking as $t$ tends to $+\infty$.

In particular, every $\mu^0$ with support satisfying \eqref{e-flock} belongs to the flocking region.
\end{theorem}

Note that, under the sufficient condition \eqref{e-flock}, according to \eqref{e-flockcont}, the size of the velocity support converges exponentially to $0$.
This result can be easily proved from corresponding results established in finite dimension in \cite{carfor,haliu} (using mean-fields limits), where the estimate \eqref{edo_XV} of Lemma \ref{lem_edo_XV} below is proved independently of the number of agents.
Hereafter, we rather use the particle flow and provide a simple proof.

Before proving Theorem \ref{c-flocksize}, we prove an auxiliary lemma giving some insight on the evolution of the size of supports.

\begin{lemma}\label{lem_edo_XV}
Given a solution $\mu$ of \eqref{kinetic_CS} (with $u\equiv 0$), for every time $t$, we define
\begin{equation*}
\begin{split}
X(t)&=\inf\left\{X\geq 0\mid \supp(\mu(t))\subset B({\bar x}(t),X)\times \R^d\right\},\\
V(t)&=\inf\left\{V\geq 0\mid \supp(\mu(t))\subset  \R^d\times B({\bar v},V)\right\},
\end{split}
\end{equation*}
The functions $X(\cdot)$ and $V(\cdot)$ are absolutely continuous, and we have
\begin{equation}\label{edo_XV}
\dot X(t)\leq V(t),\quad \dot V(t)\leq -\phi(2X(t)) V(t),
\end{equation}
for almost every $t\geq 0$.
\end{lemma}

\begin{proof}[Proof of Lemma \ref{lem_edo_XV}.]
Since displacements of the support have bounded velocities, both $X(\cdot)$ and $V(\cdot)$ are absolutely continuous functions, and hence are differentiable almost everywhere.

From Section \ref{sec_appli_CS}, and in particular from \eqref{controlledpropag_supp} (with $u\equiv 0$), the support of $\mu(t)$ is the image of the support of $\mu(0)$ under the particle flow $\Phi(t)$ at time $t$. Denoting by $(x(\cdot,x^0,v^0),v(\cdot,x^0,v^0))$ the (particle trajectory) solution of \eqref{controlledparticleflow} (with $u\equiv 0$) such that $(x(0,x^0,v^0),v(0,x^0,v^0))=(x^0,v^0)$ at time $0$, this means that $(x(t,x^0,v^0),v(t,x^0,v^0))\in\supp(\mu(t))$, for every $(x^0,v^0)\in \supp(\mu^0)$, and it follows that
\begin{equation*}
\begin{split}
X(t) &= \max \left\{  \Vert x(t,x^0,v^0)-\bar x(t)\Vert \mid   (x^0,v^0)\in \supp(\mu^0)   \right\}, \\
V(t) &= \max \left\{\Vert v(t,x^0,v^0)-\bar v\Vert \mid   (x^0,v^0)\in \supp(\mu^0)   \right\} ,
\end{split}
\end{equation*}
for every $t\geq 0$.
Note that the maximum is reached because it is assumed that $\supp(\mu^0)$ is compact.
For every $t\geq 0$, we denote by $K^X_t\subset \supp(\mu^0)$ (resp. $K^V_t\subset \supp(\mu^0)$) the set of points $(x^0,v^0)$ such that the maximum is reached in $X(t)$ (resp., in $V(t)$).

By definition, we have $X(t)^2 = \Vert x(t,x^0,v^0)-\bar x(t)\Vert^2$ for every $(x^0,v^0)\in K^X_t$, and it follows from the Danskin theorem (see \cite{Danskin}) and from the fact that $\partial_t x(t,x^0,v^0)=v(t,x^0,v^0)$ that
$$
X(t) \dot X(t) = \max \left\{ \langle x(t,x^0,v^0)-\bar x(t), v(t,x^0,v^0)-\bar v\rangle \mid (x^0,v^0)\in K^X_t \right\},
$$
and therefore, using the Cauchy-Schwarz inequality, we infer that $\dot X(t)\leq \Vert v(t,x^0,v^0)-\bar v\Vert\leq V(t)$.

Similarly, we have $V(t)^2 = \Vert v(t,x^0,v^0)-\bar v\Vert^2$ for every $(x^0,v^0)\in K^V_t$. Note that, by the first definition of $V(t)$, we have $\supp(\mu(t))\subset\R^d\times B(\bar v,\Vert v(t,x^0,v^0)-\bar v\Vert)$.
Using again the Danskin theorem and \eqref{controlledpropag_supp} (with $u\equiv 0$), we have
$$
V(t) \dot V(t) = \max \left\{ \left\langle v(t,x^0,v^0)-\bar v, \xi[\mu(t)](x(t,x^0,v^0),v(t,x^0,v^0))) \right\rangle \mid (x^0,v^0)\in K^X_t \right\} ,
$$
and, using \eqref{def_xi}, we have
\begin{multline*}
\left\langle\xi[\mu(t)](x(t,x^0,v^0),v(t,x^0,v^0)),v(t,x^0,v^0)-\bar v\right\rangle \\
= \int_{\supp(\mu(t))}\phi(\Vert x(t,x^0,v^0)-y\Vert)\langle w-v(t,x^0,v^0), v(t,x^0,v^0)-\bar v\rangle\,d\mu(t)(y,w) ,
\end{multline*}
for every $t\geq 0$. In the integral, we have $(y,w)\in\supp(\mu(t))$, and hence $w\in B(\bar v,V(t))$ and therefore $\langle w-v(t,x^0,v^0), v(t,x^0,v^0)-\bar v\rangle\leq  0$ by convexity, because $v(t,x^0,v^0)$ belongs to the boundary of the ball $B(\bar v,V(t))$, by construction. Since $\phi$ is non-increasing and $\Vert x(t,x^0,v^0)-y\Vert\leq 2X(t)$ for every $(y,w)\in\supp(\mu(t))$, we infer that
\begin{multline*}
\left\langle\xi[\mu(t)](x(t,x^0,v^0),v(t,x^0,v^0)),v(t,x^0,v^0)-\bar v\right\rangle \\
\leq \phi(2X(t)) \int_{\supp(\mu(t))} \langle w-v(t,x^0,v^0), v(t,x^0,v^0)-\bar v\rangle\,d\mu(t)(y,w) .
\end{multline*}
Since $\int_{\supp(\mu(t))}  w \, d\mu(t)(y,w) =\bar v$ and $\int_{\supp(\mu(t))}  d\mu(t)(y,w) = 1$, it follows that
$$
\left\langle\xi[\mu(t)](x(t,x^0,v^0),v(t,x^0,v^0)),v(t,x^0,v^0)-\bar v\right\rangle \\
\leq -\phi(2X(t)) V(t)^2  .
$$
Finally, we conclude that $\dot V(t) \leq -\phi(2X(t)) V(t)$.
\end{proof}

Let us now prove Theorem \ref{c-flocksize}.

\begin{proof}[Proof of Theorem \ref{c-flocksize}.]
We prove \eqref{e-flockcont}, which implies the flocking of $\mu$. Using \eqref{e-flock}, we can prove that there exists $X_M>0$ such that $X (t)\leq X_M$ and $V(t)\leq V^0 e^{-\phi(2X_M)t}$ for every $t\geq 0$, with $X(t),V(t)$ defined in Lemma \ref{lem_edo_XV}.

The reasoning is similar to the proof of \cite[Theorem 3.2]{haliu}.
Using \eqref{e-flock}, since $\phi$ is nonnegative, there exists $X_M>0$ such that $V^0<  \int_{X^0}^{X_M} \phi(2 x)\,dx$. By contradiction, let us assume that $X(T)>X_M$ for some $T\geq 0$. Using \eqref{edo_XV}, we infer that
$$
V(T)\leq V^0-\int_0^T \phi(2X(t))\dot X(t)\,dt=V^0-\int_{X(0)}^{X(T)} \phi(2x)\,dx\leq V^0-\int_{X^0}^{X_M} \phi(2 x)\,dx<0,
$$
which contradicts the fact that $V(t)\geq 0$ for every $t\geq 0$. Therefore $X (t)\leq X_M$ for every $t\geq 0$. Since $\phi$ is nonincreasing, we have $\dot V(t)\leq -\phi(2X(t))V(t) \leq  -\phi(2 X_M)V(t)$, and thus $V(t)\leq V^0 e^{-\phi(2 X_M)t}$, for every $t\geq 0$. The theorem is proved.
\end{proof}

In order to prove our main results, we will use Theorem \ref{c-flocksize} as follows.

\begin{corollary}\label{cor_flocksize}
Let $\mu^0\in \mathcal{P}_c(\R^d\times\R^d)$.
Assume that there exist $(x^0,v^0)\in\R^d\times\R^d$ and some positive real numbers $\tilde X^0$ and $\tilde V^0$ such that $\supp(\mu^0)\subset B(x^0,\tilde X^0)\times B(v^0,\tilde V^0)$. If 
\begin{equation}\label{e-flock2}
2\tilde V^0\leq \int_{2\tilde X^0}^{+\infty} \phi(2x)\,dx ,
\end{equation}
then $\mu$ converges to flocking as $t$ tends to $+\infty$.

In particular, every $\mu^0$ with support satisfying \eqref{e-flock2} belongs to the flocking region.
\end{corollary}

\begin{proof}
It suffices to note that the barycenter $({\bar x}^0,{\bar v})$ of $\mu^0$ is contained in $B(x^0,X)\times B(v^0,V)$, and hence that $\supp(\mu^0)\subseteq B({\bar x},2X)\times B({\bar v},2V)$.
\end{proof}

\section{Proof of Theorem \ref{mainthm}}\label{sec_proof_mainthm}
In this section, we prove Theorem \ref{mainthm}.

We first establish some useful estimates on the interaction kernel $\xi[\mu]$ in Section \ref{sec_4.1}, for any measure $\mu$. These technical estimates will be useful in the proof of the main theorem.

In Section \ref{sec_4.2}, we provide some general estimates on absolutely continuous solutions of \eqref{kinetic_CS}.

After these preliminaries, we focus on the proof of Theorem \ref{mainthm}. 
Given any initial condition $\mu_0$, our objective is to design a control satisfying the constraints \eqref{cont_u} and \eqref{cont_omega}, steering the system \eqref{kinetic_CS} to flocking.

The strategy that we adopt is the following. We first steer the system to the flocking region (defined in Definition \ref{def-flock}) within a finite time $T$ by means of a suitable control. This control is piecewise constant in time: we divide the time interval $[0,T]$ in subintervals $[t_k,t_{k+1})$ and the control is computed as a function of $\mu(t_k)$. After reaching the flocking region at time $T$, we switch off the control and let the uncontrolled equation \eqref{kinetic_CS} (with $u\equiv 0$) converge (asymptotically) to flocking.

The time $T$ depends on the initial distribution $\mu^0$ of the crowd: the more ``dispersed" $\mu^0$ is, the larger $T$ is.
Of course, if $\mu^0$ already belongs to the flocking region then it is not necessary to control the equation (hence $T=0$ in that case).

We proceed in two steps. 
In Section \ref{s-Cpop1}, we design an effective control $\chi_\omega u$ in the one-dimensional case $d=1$.
In Section \ref{s-SSd}, we extend the contruction to any dimension $d\geq 1$.
Section \ref{s-space} is devoted to the proof of the variant of Theorem \ref{mainthm}, with control constraints \eqref{cont_u} and \eqref{cont_omega_variante}.

\subsection{Preliminary estimates on the interaction kernel $\xi[\mu]$}\label{sec_4.1}
Let $\mu\in \mathcal{P}_c(\R^d\times\R^d)$ be arbitrary.
In this section, we study the behavior of the interaction kernel $\xi[\mu]$ defined by \eqref{def_xi}, in function of the support of $\mu$.

Recall that the space of configurations $(x,v)$ is $\R^d\times\R^d$. 
We consider the canonical orthonormal basis $(e_1,\ldots,e_{2d})$ of $\R^d\times\R^d$, in which we denote $x=(x_1,\ldots,x_d)$ and $v=(v_1,\ldots,v_d)$.

For simplicity of notation, we assume that, for every $k\in\{1,\ldots,d\}$, the $k$-th component of the spatial variable satisfies $x_k\in[0,Y_k]$, eventually after a translation in the spatial variables, where $Y_k\geq 0$ is the size of the support in the variable $x_k$.
Similarly, we assume that $v_k\in[0,W_k]$, where $W_k\geq 0$ is the size of the support in the variable $v_k$.
Note that, with this choice, we have invariance of the positive space $[0,+\infty)^d\times [0,+\infty)^d$. 

We start with an easy lemma. 
\begin{lemma}\label{lem_trivial}
Let $\mu\in\mathcal{P}_c^{ac}(\R^d\times\R^d)$ be such that $\supp(\mu)\subset\R^d\times[0,V_*]^{k-1}\times [0,W_k]\times[0,V_*]^{d-k}$ for some $V_*\geq 0$ and $W_k\geq 0$. Then, for every $(x,v)\in\R^d\times\R^d$ such that $v_k\geq W_k$ (resp., $v_k\leq 0$), we have $\langle\xi[\mu](x,v),e_{d+k}\rangle\leq 0$ (resp., $\langle\xi[\mu](x,v),e_{d+k}\rangle\geq 0$).
\end{lemma}

\begin{figure}[h]
\begin{center}
\resizebox{6cm}{!}{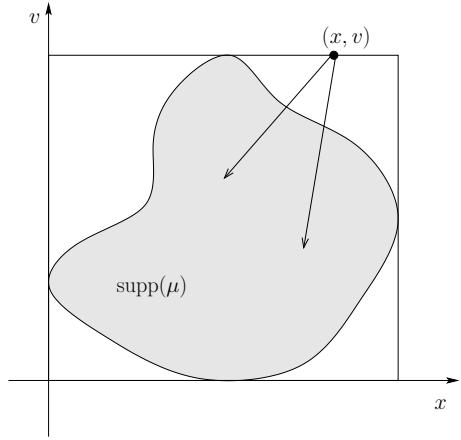}
\end{center}
\caption{Vector field $\xi[\mu]$}\label{fig-fig_supp}
\end{figure}

The lemma is obvious by using the expression
$\xi[\mu](x,v) = \iint_{\supp(\mu)} \phi(\Vert x-y\Vert) (w-v)\, d\mu(y,w)$, since $\phi$ is nonnegative and $w\in\supp(\mu)$ implies that $w_k\leq W_k$, hence $\langle(w-v), e_{d+k}\rangle=w_k-v_k\leq 0$.
Lemma \ref{lem_trivial} implies that, if $(x,v)\notin \R\times[0,V_*]^d$, then the vector field $\xi[\mu]$ is pointing inwards (see Figure \ref{fig-fig_supp}). Note that this is in accordance with the fact that the velocity part of $\supp(\mu)$ has a trend to shrink, as proved (more precisely) by the differential inequality \eqref{edo_XV} of Lemma \ref{lem_edo_XV}.

\medskip

Let us now establish a more technical result, which will be instrumental in order to prove Theorem \ref{mainthm}.

\begin{lemma}\label{p-perp}
Let $\mu\in \mathcal{P}_c(\R^d\times\R^d)$, with velocity barycenter ${\bar v}=(\bar v_1,\ldots,\bar v_d)$. We assume that there exist $\tilde x\in\R^d$, a real number $a_k$ and nonnegative real numbers $X$, $V_*$, $W_k$ such that
$$
\supp(\mu)\subset B(\tilde x,X)\times [0,V_*]^{k-1}\times[a_k,a_k+W_k]\times[0,V_*]^{k-d} .
$$ 
Let $(x,v)\in \R^d\times\R^d$ be such that $v_k-\bar v_k>r^+$ with
\begin{equation}\label{e-perp+}
r^+=\frac{\phi(0)}{\phi(0)+\phi(2X)} (W_k+a_k-\bar v_k).
\end{equation}
Then $\left\langle \xi[\mu](x,v) , (v_k-\bar v_k)e_{d+k}\right\rangle < 0$.

Similarly, let $(x,v)\in \R^d\times\R^d$ be such that $v_k-\bar v_k<-r^-$ with
\begin{equation}\label{e-perp-}
r^-=\frac{\phi(0)}{\phi(0)+\phi(2X)} (\bar v_k-a_k).
\end{equation}
Then $\left\langle \xi[\mu](x,v) , (v_k-\bar v_k)e_{d+k}\right\rangle < 0$.
\end{lemma}

\begin{proof} We prove the result with $a_k=0$ only, by observing that the case $a_k\neq 0$ can be recovered by translation of the $k$-th velocity variable. We give the proof of the first case only, in which $v_k-\bar v_k>r^+$ (for the second case, it suffices to use the change of variable $v_k\mapsto W_k-v_k$).

We want to prove that 
\begin{equation}\label{e-condr+}
\int_{\R^d\times\R^d} \phi(\Vert x-y\Vert)(w_k-v_k)(v_k-\bar v_k)\,d\mu(y,w) <0.
\end{equation}
Writing $w_k-v_k=(w_k-(\bar v_k+r^+))+((\bar v_k+r^+)-v_k)$, and noting that
\begin{equation*}
\int_{\R^d\times\R^d} \phi(\Vert x-y\Vert)((\bar v_k+r^+)-v_k)(v_k-\bar v_k)\,d\mu(y,w) <0,
\end{equation*}
since $\phi$ is nontrivial and nonnegative, $((\bar v_k+r^+)-v_k)<0$, $v_k-\bar v_k>0$ and since $\mu$ is a measure with positive mass, it follows that, to prove \eqref{e-condr+}, it suffices to prove that
\begin{equation}\label{e-condr+2}
\int_{\R^d\times\R^d} \phi(\Vert x-y\Vert)(w_k-(\bar v_k+r^+))(v_k-\bar v_k)\,d\mu(y,w) \leq 0.
\end{equation}

The space $\R^d\times\R^d$ is the union of the three (disjoint) subsets $A$, $B$ and $C$ defined by
\begin{equation*}
\begin{split}
A&=\{(y,w)\in \R^d\times\R^d\mid \bar v_k+r^+ \leq w_k\}, \\
B&=\{(y,w)\in\R^d\times\R^d\mid {\bar v}_k\leq w_k< \bar v_k+r^+\},\\
C&=\{(y,w)\in\R^d\times\R^d\mid w_k< {\bar v}_k\} .
\end{split}
\end{equation*}


Note that, since $(w_k-(\bar v_k+r^+))<0$ in $B$ and $v_k-\bar v_k>0$, we have
$$
\int_B \phi(\Vert x-y\Vert)(w_k-(\bar v_k+r^+))(v_k-\bar v_k)\,d\mu(y,w) \leq 0 .
$$
As a consequence, we will prove \eqref{e-condr+2} by establishing the (stronger) inequality
\begin{multline}\label{e-condr+4}
\int_A \phi(\Vert x-y\Vert)(w_k-(\bar v_k+r^+))(v_k-\bar v_k)\,d\mu(y,w) \\
\leq \int_C \phi(\Vert x-y\Vert)((\bar v_k+r^+)-w_k)(v_k-\bar v_k)\,d\mu(y,w).
\end{multline}
Noting that $\phi(2X)\leq \phi(\Vert x-y\Vert)\leq \phi(0)$ since $\phi$ is decreasing and $\Vert x-y\Vert\leq 2X$, and using the definitions of $A$ and of $r^+$, we get
\begin{multline*}
\int_A \phi(\Vert x-y\Vert)(w_k-(\bar v_k+r^+))(v_k-\bar v_k)\,d\mu(y,w) \\
\leq \mu(A)\phi(0) (W_k-(\bar v_k+r^+)) (v_k-\bar v_k)
=\mu(A)\frac{\phi(0) \phi(2X)}{\phi(0)+\phi(2X)}(W_k-\bar v_k)(v_k-\bar v_k) ,
\end{multline*}
and
$$
\int_C \phi(\Vert x-y\Vert)((\bar v_k+r^+)-w_k)(v_k-\bar v_k)\,d\mu(y,w)\geq \phi(2X)\int_C (\bar v_k-w_k)(v_k-\bar v_k)\,d\mu(y,w).
$$
Since $v_k-\bar v_k>r^+>0$ and $\phi(2X)>0$, to prove \eqref{e-condr+4}, it suffices to prove that
\begin{equation}\label{e-condr+5}
\mu(A)\frac{\phi(0)}{\phi(0)+\phi(2X)}(W_k-\bar v_k)=\mu(A)r^+\leq \int_C (\bar v_k-w_k)\,d\mu(y,w).
\end{equation}
By definition of the velocity barycenter ${\bar v}$ of $\mu$, we have $\int_{\R^d\times\R^d} \langle w-{\bar v}, z\rangle \,d\mu(y,w) =0$, for any $z\in\R^d$. Choosing $z=e_{d+k}$, we get that
\begin{equation}\label{e-2}
\int_{A} (w_k-{\bar v}_k)\,d\mu(y,w)+ \int_{B} (w_k-{\bar v}_k)\,d\mu(y,w)=\int_{C} ({\bar v}_k-w_k)\,d\mu(y,w).
\end{equation}
By definition of the sets $A$, $B$ and $C$, all integrals in \eqref{e-2} are nonnegative, and in particular we infer that
\begin{equation*}
\int_{A} (w_k-{\bar v}_k)\,d\mu(y,w)\leq \int_{C} ({\bar v}_k-w_k)\,d\mu(y,w).
\end{equation*}
Since $w_k-{\bar v}_k\geq r^+$ in $A$, the inequality \eqref{e-condr+5} follows. The lemma is proved.
\end{proof}

\subsection{Estimates on the solutions of (\ref{kinetic_CS}) with control}\label{sec_4.2}
Recall that the space barycenter $\bar x(t)$ and the velocity barycenter $\bar v(t)$ of $\mu(t)$ are defined by \eqref{def_space_velocity_barycenter}.
Due to the action of $\chi_\omega u$, the velocity barycenter is not constant. 
We have the following result.

\begin{lemma}\label{p-bary} 
Let $\mu\in C^0(\R,\mathcal{P}_c(\R^d\times\R^d))$ be a solution of \eqref{kinetic_CS}. We have
$$
\dot {\bar x}(t)={\bar v}(t),\qquad\dot {\bar v}(t)=\int_{\omega(t)} u(t,x,v)\, d\mu(t)(x,v) ,
$$
for every $t\in\R$.
\end{lemma}

\begin{proof}
Proceeding as in the proof of Proposition \ref{prop1}, considering \eqref{kinetic_CS} in the sense of measures, we compute
\begin{equation*}
\dot {\bar x}_k(t) =  \partial_t  \int_{\R^d\times\R^d} x_k \, d\mu(t)(x,v) 
=\int_{\R^d\times\R^d} \mathrm{div}_x  (x_k v)\, d\mu(t)(x,v) \\
=\int_{\R^d\times\R^d} v_k\, d\mu(t)(x,v)={\bar v}_k(t) ,
\end{equation*}
for every $k\in\{1,\ldots,d\}$. Similarly, using the fact that $\int_{\R^d\times\R^d} \xi[\mu]\, d\mu = 0$ (by antisymmetry), we get
$$
\dot {\bar v}_k(t)=\partial_t \int_{\R^d\times\R^d} v_k \, d\mu(t)(x,v) = \int_{\R^d\times\R^d} \chi_{\omega(t)} u_k(t,x,v) \, d\mu(t)(x,v),
$$
for every $k\in\{1,\ldots,d\}$.
\end{proof}

Let us now consider solutions $\mu(t)=f(t)\dxv$ of \eqref{kinetic_CS} that are absolutely continuous.
Let us then estimate the evolution of the $L^\infty$ norm of $f(t)$. 

\begin{lemma}\label{prop_e-Lp}
Let $\mu=f\,\dxv\in C^0(\R,\mathcal{P}_c^{ac}(\R^d\times\R^d))$ be a solution of \eqref{kinetic_CS}, with a Lipschitz control $\chi_{\omega} u$.
For every $p\in[ 1,+\infty]$, we have the estimate
\begin{equation}\label{e-Lp}
\frac{d}{dt} \Vert f(t,\cdot,\cdot)\Vert_{L^p(\R^d\times\R^d)}\leq\frac{p-1}{p} \Vert f(t,\cdot,\cdot)\Vert_{L^p(\R^d\times\R^d)}  \left( \phi(0)d+\Vert \mathrm{div}_v (u(t,\cdot,\cdot))\Vert _{L^\infty(\omega(t))} \right),
\end{equation}
for every $t\in\R$, with the agreement that $\frac{p-1}{p}=1$ for $p=+\infty$.
\end{lemma}

\begin{proof}
The proof is a generalization of the proof of \cite[Proposition 3.1]{hatad}. Using \eqref{kinetic_CS}, we have
\begin{multline}\label{e-3}
\frac{d}{dt} \int_{\R^d\times\R^d} f^p \dxv = p \int_{\R^d\times\R^d}  -f^{p-1} \langle v, \mathrm{grad}_x f\rangle \dxv - p \int_{\R^d\times\R^d} f^{p-1}\mathrm{div}_v(\xi[f]f)\dxv \\
- p \int_{\R^d\times\R^d} f^{p-1} \mathrm{div}_v (\chi_\omega u f)\dxv.
\end{multline}
Let us compute the three terms at the right-hand side of \eqref{e-3}. 
The first term is equal to $-\int_{\R^d\times\R^d} \mathrm{div}_x ( f^{p} v) \dxv$ and hence is equal to $0$ since $f^{p}$ has compact support.
For the second term, noting that $f^{p-1} \mathrm{div}_v (\xi[f] f) = f^{p} \mathrm{div}_v ( \xi[f]) + f^{p-1} \langle \xi[f], \nabla_v f\rangle$, and that
$$
\mathrm{div}_v \left( \xi[f] f^p\right) = \mathrm{div}_v (\xi[f]) f^p  + \left\langle \xi[f], \nabla_vf^p \right\rangle = \mathrm{div}_v (\xi[f]) f^p + p f^{p-1} \langle\xi[f], \nabla_v f\rangle,
$$
we infer that
$$
p f^{p-1} \mathrm{div}_v(\xi[f] f) = (p-1) f^{p}\, \mathrm{div}_v (\xi[f]) + \mathrm{div}_v \left( \xi[f] f^p \right) .
$$
It follows that
\begin{equation*}
\begin{split}
p \left\vert \int_{\R^d\times\R^d} f^{p-1}\mathrm{div}_v (\xi[f]f)\dxv \right\vert
& = (p-1)\left\vert \int_{\R^d\times\R^d} f^p \, \mathrm{div}_v (\xi[f])\dxv \right\vert \\
& \leq  (p-1) \Vert f\Vert_{L^p(\R^d\times\R^d)}^p \Vert \mathrm{div}_v (\xi[f]) \Vert _{L^\infty(\supp(f))}.
\end{split}
\end{equation*}
Similar estimates are done for the third term by replacing $\xi[f]$ with $\chi_\omega u$, that is a Lipschitz vector field. Using \eqref{e-3}, we get 
$$
\frac{d}{dt} \Vert f\Vert_{L^p(\R^d\times\R^d)} \leq \frac{p-1}{p} \Vert f\Vert_{L^p(\R^d\times\R^d)}  \left( \Vert \mathrm{div}_v (\xi[f]) \Vert _{L^\infty(\supp(f))}+\Vert \mathrm{div}_v (u)\Vert _{L^\infty(\omega(t))}  \right).
$$
Finally, noting that
$$
\left\vert \partial_{v_k} \int_{\R^d\times\R^d}\phi(\Vert x-y\Vert) (w_k-v_k)f(y,w)\,dy\,dw \right\vert
= \left\vert - \int_{\R^d\times\R^d}\phi(\Vert x-y\Vert)f(y,w)\,dy\,dw\right\vert 
\leq \phi(0),
$$ 
for every $k\in\{1,\ldots,d\}$, it follows that $\Vert \mathrm{div}_v (\xi[f]) \Vert _{L^\infty(\supp(f))} \leq \phi(0) d$, and this yields \eqref{e-Lp}.
\end{proof}


\subsection{Proof of Theorem \ref{mainthm} in the one-dimensional case}\label{s-Cpop1}
Throughout this section, we assume that $d=1$.

We first define the {\em fundamental step} $\SS$ of our algorithm in Section \ref{s-ss}. We prove in Section \ref{s-sss} that a finite number of iterations of this fundamental step $\SS$ provides convergence to flocking.

\subsubsection{Fundamental step $\SS$}\label{s-ss}
Hereafter, we define the fundamental step $\SS$ of our strategy. The strategy takes, as an input, a measure $\mu^0=\mu(0)$ (absolutely continuous) standing for the initial data of \eqref{kinetic_CS}, and provides, as outputs, a time $T^0$ and a measure $\mu^1=\mu(T^0)$ (which will be proved to be absolutely continuous), standing for the time horizon and the corresponding solution of \eqref{kinetic_CS} at time $T^0$ for some adequate control $\chi_\omega u$.


In the definition below, the bracket subscript stands for the index of a given sequence. It is used in order to avoid any confusion with coordinates subscripts.

\paragraph{Definition of the control $\chi_\omega u$ along the time interval $[0,T^0]$ (fundamental step $\SS$).} 
In order to define the control, we need to define, at every time $t$, the control set $\omega(t)$ (on which the control acts), and the control force $u(t,x,v)$ for every $(x,v)\in\omega(t)$.
We are actually going to set
$$
u(t,x,v) = \psi(t,x,v) \frac{v-\bar v(0)}{\vert v-\bar v(0)\vert},
$$
for every $t\in [0,T^0]$, and every $(x,v)\in\R^d\times\R^d$, where the function $\psi$, constructed below, is piecewise constant in $t$ for $(x,v)$ fixed, continuous and piecewise linear in $(x,v)$ for $t$ fixed (see Figure \ref{fig-phi}), and where the control set $\omega(t)$ is piecewise constant in $t$.

\medskip

Since the construction of the control is quite technical, we first provide an intuitive idea of how to define it. According to Lemma \ref{p-perp}, the set $\R\times [\bar v(t)-r^-(t),\bar v(t)+r^+(t)]$ (whose precise definition is given below) is invariant under the particle flow dynamics, and therefore, inside this invariant set, it is not useful to act, and hence we set $u=0$ there. Outside of that set, we want to push the population inwards. Since the invariant set is variable in time, we make precise estimates to have a larger set that is invariant on the whole interval $[0,T^0]$. Since the population outside of such a set can have a mass larger than the constraint $c$, due to the control constraint \eqref{cont_u} it is not possible to act on that population in its whole at any time $t$, and our strategy consists of splitting the domain into ``slices'' $\Omega_i(t)$ such that each slice contains a mass $\frac{c}{2}$, and then we will act on each of those slices, on successive small time intervals. With precise estimates on the displacement of mass, we will then check that $\Omega_i(t)$ satisfies the constraint $\mu(\Omega_i(t))\leq c$ for every $t\in[0,T^0]$.
 
\medskip

Let us now give the precise definition of the control.

Let $\mu^0=f^0\,\dxv\in \mathcal{P}_c^{ac}(\R\times\R)$ be an initial datum. Using a translation, we assume that $\supp(\mu^0)\subset [0,Y^0]\times [0,W^0]$, where $Y^0\geq 0$ is the size of the support in the variable $x$ and $W^0\geq 0$ is the size of the support in the variable $v$. Admitting temporarily that, with the control that we will define, there exists a unique solution $\mu$ of \eqref{kinetic_CS} such that $\mu(0)=\mu^0$, which is absolutely continuous, we assume that $\supp(\mu(t))\subset [0,Y(t)]\times [a(t),a(t)+W(t)]$, where 
\begin{equation*}
\begin{split}
Y(t)&=\max \left\{ \vert x\vert \mid (x,v)\in \supp(\mu(t))\right\},\\
a(t)&=\min \left\{ \vert v\vert \mid (x,v)\in \supp(\mu(t))\right\},\\
W(t)&=\max \left\{ \vert v\vert \mid (x,v)\in \supp(\mu(t))\right\}-a(t),
\end{split}
\end{equation*}
with $Y(0)=Y^0$, $W(0)=W^0$, and $a(0)=0$.

Let ${\bar v}(t)={\bar v}(t)\in(a(t),a(t)+W(t))$ be the velocity barycenter of $\mu(t)$ (note that $a(t)<{\bar v}(t)<a(t)+W(t)$, with a strict inequality because $\mu(t)$ is absolutely continuous). We set $\bar v^0=\bar v(0)$.

For $T^0>0$ to be chosen, we are going to define the control along the time interval $[0,T^0]$. We consider a regular subdivision of the time interval $[0,T^0]$, into $n$ subintervals,
$$
[0,T^0] = \bigcup_{i=1}^n \left[ \frac{(i-1)T^0}{n},\frac{iT^0}{n} \right),
$$
and, along each time subinterval $\left[ \frac{(i-1)T^0}{n},\frac{iT^0}{n} \right)$, we set $\omega(t) = \omega_{[i]}$ and $\psi(t,x,v) = \psi_{[i]}(x,v)$, with $\omega_{[i]}$ and $\psi_{[i]}$ defined as follows.

\medskip

First of all, we define the functions 
\begin{equation*}
\begin{split}
\alpha^+(t)&=\frac{\phi(0)}{\phi(0)+\phi(Y(t)+W(t))}(W(t)+a(t)-{\bar v}(t)),\\
\beta^+(t)&=\frac{1}{3} (W(t)+a(t)-\alpha^+(t)-{\bar v}(t))=\frac{1}{3}\frac{\phi(Y(t)+W(t))}{\phi(0)+\phi(Y(t)+W(t))}(W(t)+a(t)-{\bar v}(t)),\\
\alpha^-(t)&=\frac{\phi(0)}{\phi(0)+\phi(Y(t)+W(t))}({\bar v}(t)-a(t)),\\
\beta^-(t)&=\frac{1}{3} ({\bar v}(t)-a(t)-\alpha^-(t))=\frac{1}{3} \frac{\phi(Y(t)+W(t))}{\phi(0)+\phi(Y(t)+W(t))}({\bar v}(t)-a(t)),\\
\alpha(t)&=\max\left(\alpha^+(t),\alpha^-(t)\right),\\
\beta(t)&=\max\left(\beta^+(t),\beta^-(t)\right),
\end{split}
\end{equation*}
and we set $\alpha^0=\alpha(0)$ and $\beta^0=\beta(0)$. Let us now define the control sets $\omega_{[i]}$, $i=1,\ldots,n$.
We write
$$
\left[ 0,Y^0 \right] = \bigcup_{i=0}^{n-1} \left[ x_{[i]},x_{[i+1]} \right] ,
$$
with $n=\left\lceil\frac{2}{c}\right\rceil$ (integer part), and where we have set $x_{[0]}=0$, $x_{[n]}=Y^0$, and $x_{[i]}$ is defined as the minimal value such that the set $\Omega_{[i]}^0=[x_{[i-1]},x_{[i]}]\times [0,W^0]$ satisfies $\mu^0(\Omega_{[i]}^0)=\frac{c}{2}$, for every $i\in\{1,\ldots,n-1\}$. Note that the set $\Omega_{[n]}^0=[x_{[n-1]},x_{[n]}]\times [0,W^0]$ is such that $\mu^0(\Omega_{[n]}^0)\leq \frac{c}{2}$.

At this step, it is important to note that the points $x_{[i]}$ are well defined because $\mu^0\in\mathcal{P}^{ac}_c(\R\times\R)$, and thus the mass of $\mu^0$ in each interval $[x_{[i-1]},x]\times[0,W^0]$ is a continuous function of $x$.

Let $\varepsilon^0>0$ be the largest positive real number such that
$$
\mu^0([x_{[i]}-3\varepsilon^0,x_{[i+1]}+3\varepsilon^0]\times [0,W^0])\leq c, \quad\forall i\in\{1,\ldots,n\}.
$$
As above, it is important to note that $\varepsilon^0$ is well defined because $\mu^0$ is absolutely continuous.

We define now the (positive) time $T^0$ by 
$$
T^0 = \min\left( \frac{\varepsilon^0}{W^0},\frac{\beta^0}{2c},1 \right) .
$$

Now, for every $i\in\{1,\ldots,n\}$, we define (see Figure \ref{fig-phi})
\begin{multline*}
\omega_{[i]} = \left[ x_{[i]}-2\varepsilon^0,x_{[i+1]} +2\varepsilon^0 \right] \\
\times \left(  [{\bar v}^0-\alpha^0-4\beta^0,{\bar v}^0-\alpha^0-\beta^0]\cup [{\bar v}^0+\alpha^0+\beta^0,{\bar v}^0+\alpha^0+4\beta^0] \right) .
\end{multline*}

\begin{figure}[h]
\begin{center}
\includegraphics[width=8cm]{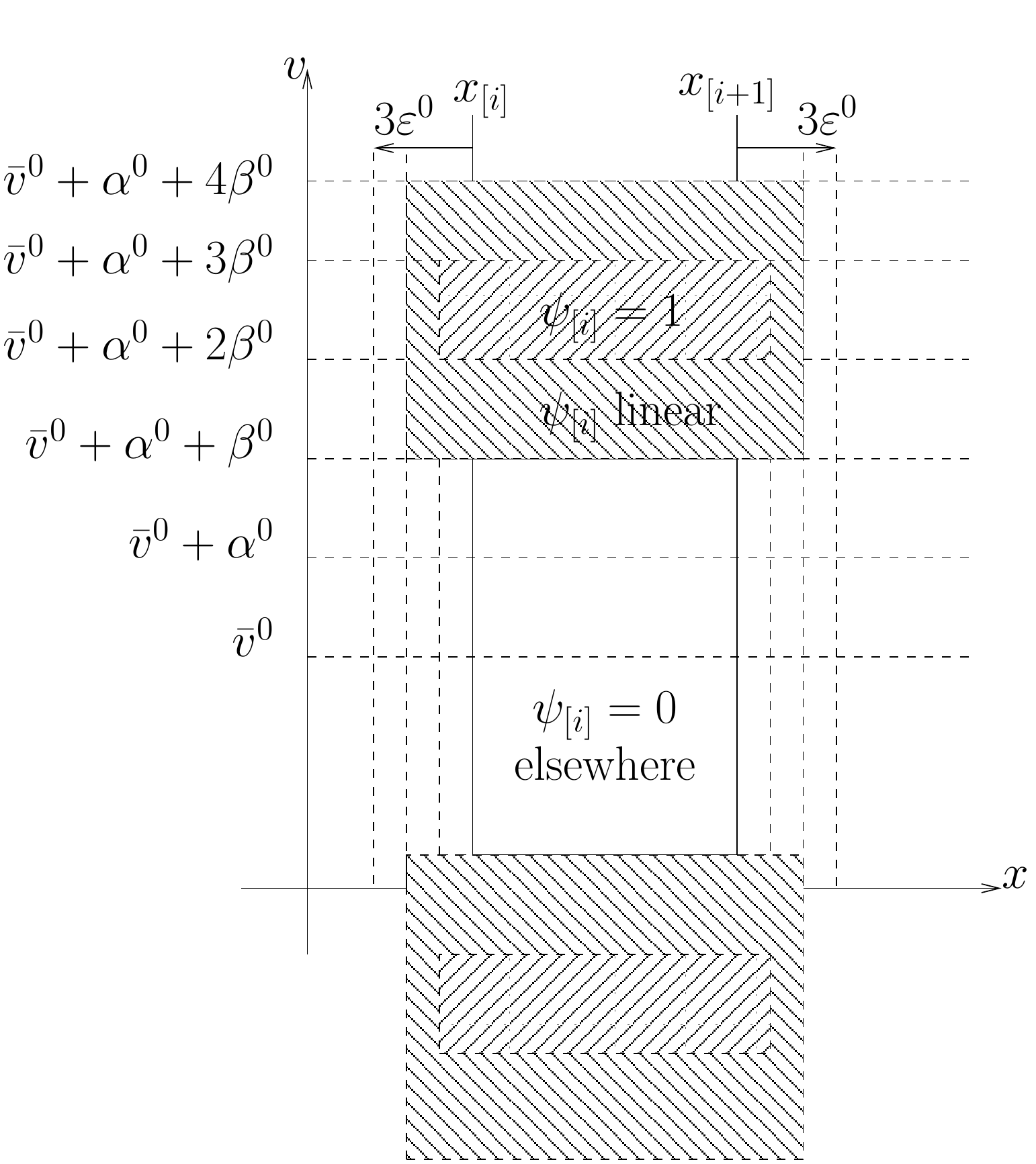}
\caption{Definition of $\psi_{[i]}$.}\label{fig-phi}
\end{center}
\end{figure}

It remains to define the functions $\psi_{[i]}$, $i=1,\ldots,n$, on $\R\times\R$.
We define the subsets of $\R\times\R$:

$A_{[i]}=\R\times\R\setminus \omega_{[i]}$, and
\begin{multline*}
B_{[i]}= [x_{[i]}-\varepsilon^0,x_{[i+1]}+\varepsilon^0] \\
\times \left( [{\bar v}^0-\alpha^0-3\beta^0,{\bar v}^0-\alpha^0-2\beta^0] \cup [{\bar v}^0+\alpha^0+2\beta^0,{\bar v}^0+\alpha^0+3\beta^0] \right) .
\end{multline*}
We set $\psi_{[i]}=0$ on $A_{[i]}$, we set $\psi_{[i]}=1$ on $B_{[i]}$, and
\begin{multline*}
\psi_{[i]}(x,v)=
\min\bigg( \frac{|x-x_{[i]}-2\varepsilon^0|}{\varepsilon^0}, \frac{|x-(x_{[i+1]}+2\varepsilon^0)|}{\varepsilon^0}, \frac{|v-({\bar v}^0-\alpha^0-\beta^0)|}{\beta^0}, \\
\frac{|v-({\bar v}^0+\alpha^0+\beta^0)|}{\beta^0}, \frac{|v-({\bar v}^0-\alpha^0-4\beta^0)|}{\beta^0}, \frac{|v-({\bar v}^0+\alpha+4\beta)|}{\beta^0} \bigg)
\end{multline*}
on $\R^2\setminus(A_{[i]}\cup B_{[i]})$.
The construction is illustrated on Figure \ref{fig-phi}.


\begin{remark}\label{rem_Ben}
For every $t\in[0,T^0]$, and for every $i\in\{1,\ldots,n\}$, we define the set $\Omega_{[i]}(t)$ as the image of the rectangle $\Omega_{[i]}^0$ under the controlled particle flow \eqref{controlledparticleflow}.
With the control strategy defined above, we note that the variables $x_{[i]}$ are defined at time $0$ only, since the sets $\Omega_{[i]}(t)$ are not rectangles anymore in general for positive times. 
\end{remark}

\begin{remark} \label{r-ck2} In connection with Remark \ref{r-ck} about higher regularity of the solution, one can easily adapt the definition of $\chi_\omega u$ to preserve regularity in the following sense. Let $\mu_0\in \mathcal{P}^{ac}_c(\R^d\times\R^d)$ such that its density  is a function of class $C^k$ on $\R^d\times\R^d$. Then, define the sets $\omega_{[i]},A_{[i]},B_{[i]}$ as before. Define $\phi_{[i]}\in C^{k-1}(\R^d\times\R^d)$ such that $\psi_{[i]}=0$ on $A_{[i]}$, $\psi_{[i]}=1$ on $B_{[i]}$ and $\|\phi_{[i]}\|_{Lip}\leq 2\eps^0$; this is possible since $A_{[i]}$ and $B_{[i]}$ are disjoint with distance $\eps^0$. Then, by applying the same strategy given below, one finds the same results and estimates given below, eventually replacing $\eps^0$ with $2\eps^0$.
\end{remark}

Let us now prove that the control $\chi_\omega u$ constructed above along the time interval $[0,T^0)$ is Lipschitz and satisfies the constraints \eqref{cont_u} and \ref{cont_omega}. We then prove that the solution $\mu$ of \eqref{kinetic_CS}, starting at $\mu^0$ at time $t=0$ and corresponding to the control $\chi_\omega u$, is well defined over the whole interval $[0,T^0]$, is absolutely continuous (that is, does not develop any singularity), and that its velocity support decreases.

\begin{lemma}\label{p-SS} 
Let $\mu^0=f^0\, \dxv\in\mathcal{P}^{ac}_c(\R\times\R)$, with compact support contained in $[0,Y^0]\times [0,W^0]$.
There exists a unique solution $\mu\in C^0([0,T^0],\mathcal{P}(\R\times\R))$ of \eqref{kinetic_CS},  corresponding to the control $\chi_\omega u$ defined by $\SS$. Moreover:
\begin{itemize}
\item $\mu\in C^0([0,T^0],\mathcal{P}^{ac}_c(\R\times\R))$, that is, the solution $\mu$ remains, like $\mu^0$, absolutely continuous and with compact support; in particular, at time $T^0$, we have $\mu^1=\mu(T^0)\in\mathcal{P}^{ac}_c(\R\times\R)$;
\item setting $Y^1=Y(T^0)$, $a^1=a(T^0)$ and $W^1=W(T^0)$, such that $\supp(\mu^1)\subset [0,Y^1]\times [a^1,a^1+W^1]$, it holds
$Y^1\leq Y^0+W^0$;
\item the domain $\R\times[{\bar v}^0-\alpha^0-\beta^0-k^-,{\bar v}^0+\alpha^0+\beta^0+k^+]$ is invariant under the controlled particle flow $\Phi_{\omega,u}(t)$ (defined in Corollary \ref{cor_appli_CS_control}), for all $k^-\geq 0$ and $k^+\geq 0$;
\item either $[a^1,a^1+W^1]\subset [0,W^0-\frac{T^0}{n}]$ or $[a^1,a^1+W^1]\subset [\frac{T^0}{n},W^0]$, which implies that $W^1\leq W^0-\frac{T^0}{n}$;
\item $\displaystyle 12\varepsilon^0\Vert f^0\Vert_{L^\infty(\R^d\times\R^d)} W^0\geq c$;
\item the control satisfies the constraints \eqref{cont_u} and \eqref{cont_omega}.
\end{itemize}
\end{lemma}

\begin{proof}
By construction, for every $i\in\{1,\ldots,n\}$, the function $\psi_{[i]}$ is Lipschitz and piecewise $C^\infty$. Therefore, the vector fields $V_{[0]}=(v,\xi[f])$ and $V_{[i]}=(v,\xi[\mu]+\chi_{\omega_{[i]}}u_{[i]})$, $i=1,\ldots,n$, are regular enough to ensure existence and uniqueness of the solution $\mu$ of \eqref{kinetic_CS} over the whole interval $[0,T^0]$. Indeed, it suffices to apply Theorem \ref{t-esistenza} iteratively over each time subinterval $\left[ \frac{(i-1)T^0}{n},\frac{iT^0}{n} \right)$ (with initial datum $\mu((i-1)T^0/n)$), and moreover, the solution $\mu$ remains absolutely continuous and with compact support.

\medskip

We claim that the domain $\R\times[0,W^0]$ is invariant under the controlled particle flow $\Phi_{\omega,u}$. Indeed, the vector fields $\xi(\mu)$ and $u_{[i]}$ (by construction) always point inwards along the boundary of that domain. Since $T^0\leq 1$ by definition, and since $\supp(\mu^0)\subset [0,Y^0]\times[0,W^0]$, it follows that $\supp(\mu(t))\subset [0,Y^0+W^0]\times[0,W^0]$ for every $t\in [0,T^0]$. 

\medskip

Let $k^-$ and $k^+$ be arbitrary nonnegative real numbers. Let us prove that the domain $D_{k^-,k^+} = \R\times[{\bar v}^0-\alpha^0-\beta^0-k^-,{\bar v}^0+\alpha^0+\beta^0+k^+]$ is invariant under the flow $\Phi_{\omega,u}$.
To this aim, it suffices to prove that the velocity vector $\xi[\mu(t)]$ points inwards along the boundary of $D_{k^-,k^+}$, that is, since we are in dimension one,
\begin{eqnarray*}
\displaystyle&&  \xi[\mu(t)](x,{\bar v}^0+\alpha^0+\beta^0+k^+)+\chi_{\omega_{[i]}}(x,{\bar v}^0+\alpha^0+\beta^0+k^+)u_{[i]}(t,x,{\bar v}^0+\alpha^0+\beta^0+k^+) <0, \\ 
\displaystyle&&  \xi[\mu(t)](x,{\bar v}^0-\alpha^0-\beta^0-k^-)+\chi_{\omega_{[i]}}(x,{\bar v}^0-\alpha^0-\beta^0-k^-)u_{[i]}(t,x,{\bar v}^0-\alpha^0-\beta^0-k^-) >0, 
\end{eqnarray*}

We start with the case $k^-=k^+=0$. 
First of all, note that $D_{0,0}\cap\omega_{[i]}=\emptyset$, which means that the control does not act on $D_{0,0}$. Then, it suffices to prove that $\xi[\mu(t)](x,{\bar v}^0+\alpha^0+\beta^0) <0$ and that $\xi[\mu(t)](x,{\bar v}^0-\alpha^0-\beta^0) >0$, for every $t\in[0,T^0]$. To this aim, we first study the evolution of ${\bar v}(t)$. Since $\dot {\bar v}=\int_{\omega} u$ (see Lemma \ref{p-bary}), $\left\vert\int_\omega u\right\vert\leq c$ and $T^0\leq \frac{\beta^0}{2c}$, we get that $\vert {\bar v}(t)-{\bar v}^0\vert \leq \frac{\beta^0}{2}<\beta(t)$. Let $t\in[0,T^0]$ be arbitrary. We assume that ${\bar v}(t)\geq {\bar v}^0$ (the case ${\bar v}(t)\leq {\bar v}^0$ is treated similarly). 
We now make use of Lemma \ref{p-perp}. First, noting that $\supp(\mu(t))\subset [0,Y^0+W^0]\times[0,W^0]$ for every $t\in [0,T^0]$, it follows that the scalar number $r^+$, defined by \eqref{e-perp+}, is equal to $\alpha^+(0)\leq \alpha^0$. Similarly, the scalar number $r^-$, defined by \eqref{e-perp-}, is equal to $\alpha^-(0)\leq \alpha^0$. Both functions $r^+(t)$ and $r^-(t)$ are constant in time along $[0,T^0]$, since the size of the domain has been estimated with constants along that time interval. Now, since ${\bar v}^0+\alpha^0+\beta^0-\bar v(t)\geq \alpha^0\geq r^+$, Lemma \ref{p-perp} implies that $\xi[\mu(t)](x,{\bar v}^0+\alpha^0+\beta^0) <0$. Similarly, since ${\bar v}^0-\alpha^0-\beta^0-\bar v(t)\leq -\alpha^0\leq -r^-$, Lemma \ref{p-perp} implies that $\xi[\mu(t)](x,{\bar v}^0-\alpha^0-\beta^0)>0$.

Similar arguments yield invariance of all domains $\R\times[{\bar v}^0-\alpha^0-\beta^0-k^-,{\bar v}^0+\alpha^0+\beta^0+k^+]$ for arbitrary $k^-\geq 0$ and $k^+\geq 0$. Indeed, we have the same properties of the vector field $\xi[f]$ pointing inwards, with the control $\chi_{\omega_{[i]}}u_{[i]}$ (when it is nonzero) pointing inwards as well.

\medskip

Let us now prove that either $[a^1,a^1+W^1]\subset [0,W^0-\frac{T^0}n]$ or $[a^1,a^1+W^1]\subset [\frac{T^0}n,W^0]$.

We first assume that $\beta^0=\beta^+(0)$. Recall that the set $\Omega_{[i]}(t)$ is defined in Remark \ref{rem_Ben}. 
We distinguish between three cases, according to whether $t\in\left[0,i\frac{T^0}n\right)$, or $t\in\left[i\frac{T^0}n,(i+1)\frac{T^0}n\right)$, or $t\in\left[(i+1)\frac{T^0}n),T^0\right]$.

For every $t\in\left[0,i\frac{T^0}n\right)$, noting that the set $\R\times[0,W^0]$ is invariant, then $\Omega^0_{[i]}\subset \R\times [0,W^0]$ implies that $\Omega_{[i]}\left(i\frac{T^0}{n}\right)\subset \R\times [0,W^0]$.

For every $t\in\left[i\frac{T^0}n,(i+1)\frac{T^0}n\right)$, we define
$$b(t)=\sup \left\{v\in\R \mid (x,v)\in\Omega_{[i]}(t)\right\}$$
so that $\Omega_{[i]}(t)\subset\R\times[a(t),b(t)]$ with $0\leq a(t)\leq b(t)\leq W^0$.
Note that $a(t)\geq 0$ for every $t\in \left[i\frac{T^0}n,(i+1)\frac{T^0}n\right)$ since the set $\R\times [0,W^0]$ is invariant. For $b(t)$, we have two cases.
\begin{itemize}
\item Either $b(t)\leq W^0-\beta^0$ for some $t\in \left[i\frac{T^0}n,(i+1)\frac{T^0}n\right)$. In this case, the set $\R\times [0,{\bar v}^0+\alpha^0+2\beta^0]$ is invariant since both sets $\R\times [0,W^0]$ and $\R\times [{\bar v}^0-\alpha^0-\beta^0,{\bar v}^0+\alpha^0+\beta^0+k^+]$ with $k^+=\beta^0$ are invariant and hence their intersection is invariant as well. Then 
$$
b\left((i+1)\frac{T}n\right)\leq W^0-\beta^0\leq W^0-2T^0c\leq W^0-4\frac{T^0}{n}<W^0-\frac{T^0}{n}.
$$
\item Or $b(t)\geq W^0-\beta^0$ on the whole interval $\left[i\frac{T^0}{n},(i+1)\frac{T^0}{n}\right)$. Since $\beta^0=\beta^+(0)$, we have $W^0-{\bar v}^0\geq {\bar v}^0$, which implies that $\alpha^0=\alpha^+(0)$, and hence that $b(t)\geq {\bar v}^0+\alpha^0+2\beta^0$. Let $(x,v)\in\Omega_{[i]}(t)$ be such that $v=b(t)$. Note that $(x,v)\in\Omega_{[i]}(t)$ implies that $d(x,\Omega_{[i]}^0)\leq W^0 T^0\leq \varepsilon^0$. We also have $v=b(t)\geq {\bar v}^0+\alpha^0+2\beta^0$. The two conditions imply that $\psi_{[i]}(x,v)=1$, which in turn implies that $\chi_{\omega_{[i]}} u_{[i]} (x,v)=-1$. Then, the velocity component of the vector field acting on $(x,v)$ is $\xi[\mu(t)]-1$. Recall that $\xi[\mu(t)](x,v)<0$ because $v-{\bar v}(t)>\alpha^0$. Since this estimate holds for any $(x,v)\in\Omega_{[i]}(t)$ with $v=b(t)$, then $\dot b(t)<- 1$. Since this holds on the whole interval $\left[i\frac{T^0}{n},(i+1)\frac{T^0}n\right)$ and $b\left(i\frac{T^0}{n}\right)\leq W^0$, then $b\left((i+1)\frac{T^0}{n}\right)\leq W^0-\frac{T^0}{n}$.
\end{itemize}
In both cases, we have obtained that $\Omega_{[i]}\left((i+1)\frac{T^0}{n}\right)\subset \R\times \left[0, W^0-\frac{T^0}{n}\right]$.

Finally, for every $t\in\left[(i+1)\frac{T^0}{n},T^0\right]$, since the set $\R\times\left[ 0, W^0-\frac{T^0}{n} \right]$ is invariant, it follows that $\Omega_{[i]}(T^0)\subset \R\times\left[0, W^0-\frac{T^0}{n}\right]$.

Since the estimate holds for all sets $\Omega_{[i]}(t)$, we conclude that the support of $\mu^1=\mu(T^0)$ is contained in $\R\times\left[0, W^0-\frac{T^0}{n}\right]$.

The case where $\beta^0=\beta^-(0)$ is similar, by proving that that $a(T^0)\geq \frac{T^0}{n}$ and $b(t)\leq W^0$.

\medskip

Let us now prove that $\varepsilon^0\geq \frac{c}{12\Vert f^0\Vert _\infty W^0}$.
Consider the mass contained in the set $[x_{[i]}-\ell,x_{[i+1]}+\ell]\times[0,W^0]$, for $\ell\geq 0$. Since the mass contained in $[x_{[i]},x_{[i+1]}]\times[0,W^0]$ is equal to $\frac{c}{2}$, then, with a simple geometric observation, it is clear that the mass contained in $[x_{[i]}-\ell,x_{[i+1]}+\ell]\times[0,W^0]$ is less than or equal to $\frac{c}{2}+2\Vert f^0\Vert _\infty \ell W^0$. Since we want to keep a mass less than or equal to $c$ (this is the control constraint), we need to have $\ell\geq \frac{c}{4\Vert f^0\Vert _\infty W^0}$. Then, we choose $3\varepsilon^0=\ell$.

\medskip

Let us finally prove the last item of the lemma.
The regularity of $\chi_{\omega} u$ is obvious, since $u$ is piecewise constant with respect to $t$ and it is Lipschitz and piecewise $C^\infty$ with respect to $(x,v)$. The constraint \eqref{cont_u} is satisfied by definition of $\psi_{[i]}$. To prove that the constraint \eqref{cont_omega} is satisfied, let us establish the stronger condition $\int_{\omega_{[i]}}f(t)\, \dxv\leq c$ for every $i\in\{1,\ldots,n\}$, where $\mu(t)=f(t)\, \dxv$. Since $\dot x(t)=v(t)\leq W^0$ for every $t\in[0,T^0]$, it follows that the mass can travel along the $x$ coordinate with a distance at most $T^0 W^0\leq \varepsilon^0$. Hence, we have
$$
\int_{\omega_{[i]}}f(t)\,\dxv=\int_{x_{[i]}-2\varepsilon^0}^{x_{[i+1]}+2\varepsilon^0}\int f(t)\, dv\, dx\leq \int_{x_{[i]}-3\varepsilon^0}^{x_{[i+1]}+3\varepsilon^0} \int f^0\, dv\, dx=c.
$$
The lemma is proved.
\end{proof}

\subsubsection{Complete strategy $\SSS$}\label{s-sss}
The complete strategy consists of repeating the fundamental step $\SS$, until reaching a prescribed size $\eta$ of the velocity support. We will then choose $\eta$ satisfying the estimate of Corollary \ref{cor_flocksize}, which ensures flocking.

\paragraph{Complete strategy $\SSS$.}
Let $\mu^0=f^0\,\dxv\in\mathcal{P}^{ac}_c(\R\times\R)$ be such that $\supp(\mu^0) \subset [0,Y^0]\times [0,W^0]$, and let $\eta>0$. We apply the fundamental step $\SS$ iteratively, replacing the superscript $0$ by the superscript $i$: while $W^i> \eta$, we compute $\mu^{i+1}=\mu(\sum_{j=0}^i T^i)$.

\smallskip

Hereafter, we prove that the above iteration terminates.

\begin{lemma}\label{p-SSeta} 
Let $\mu^0=f^0\,\dxv\in\mathcal{P}^{ac}_c(\R\times\R)$ be such that $\supp(\mu^0) \subset [0,Y^0]\times [0,W^0]$, and let $\eta>0$. Then there exists $k\in\N^*$ such that the probability measure $\mu^k=f^k\,\dxv$ computed with the complete strategy $\SSS$, with support contained in $[0,Y^k]\times [a^k,a^k+W^k]$, is such that $W^k\leq \eta$. Moreover, we have $\mu^k=\mu(\sum_{j=0}^k T^j)$ with $\sum_{j=0}^k T^j\leq W^0 \lceil \frac{2}{c}\rceil$, and we have $Y^k\leq Y^0+(W^0)^2 \lceil \frac{2}{c}\rceil$. Furthermore, the control satisfies the constraints \eqref{cont_u} and \eqref{cont_omega}.
\end{lemma}

\begin{proof}
Let us first prove that the iteration terminates. Assuming that we are at the step $i$ of the iteration, consider the real numbers $\beta^i$, $\varepsilon^i$, $W^i$ and $T^i$ obtained by applying the fundamental step $\SS$ to $\mu^i$. From Lemma \ref{p-SS}, we have $W^{i+1}\leq W^i-\frac{T^i}{n}$. Since $W^j\geq 0$ for every $j$, we have $\sum_{j=1}^i T^j\leq nW^0$ for every $i$. We set $\bar{T}=\sum_{j=1}^\infty T^j$; note that $\bar T\leq nW^0$. 
It follows that $Y^i\leq Y^0+n(W^0)^2$ for every $i$.

The sequence $(W^i)_{i\in\N^*}$ is nonnegative, bounded above (by $W^0$), and is decreasing (since $T^i>0$), therefore it converges to some $\bar W\geq 0$. Let us prove that $\bar W=0$. By contradiction, let us assume that $\bar W>0$. For any given $i$, we have either $W^i-{\bar v}^i\geq \frac{W^i}{2}\geq \frac{\bar{W}}{2}$ or ${\bar v}^i\geq \frac{W^i}{2}\geq \frac{\bar{W}}{2}$. In both cases we have 
$$
\beta^i\geq\frac13 \frac{\phi(Y^i+W^i)}{\phi(0)+\phi(Y^i+W^i)} \frac{\bar{W}}{2}\geq \frac{\phi(Y^0+n(W^0)^2+W_0)}{\phi(0)+\phi(\bar W)} \frac{\bar{W}}{2},
$$
where we have used that $0\leq Y^i\leq Y^0+n(W^0)^2$, that $\bar W\leq W^i\leq W^0$ and that $\phi$ is decreasing. Since the estimate does not depend on $i$, we have obtained that $\beta^i\geq \bar\beta$ for every $i$, with $\bar \beta=\frac{\phi(Y^0+n(W^0)^2+W^0)}{\phi(0)+\phi(\bar W)} \frac{\bar{W}}{2}>0$.
Recalling that $\mu(t)=f(t)\,\dxv$, let us consider the function $t\mapsto\Vert f(t)\Vert _\infty$ on the interval $[0,\bar T)$ (note that the interval is open at $\bar T$ because we have not yet proved the convergence of the complete strategy). Using the definition of $\psi^i$, we get
$$
\left\Vert \mathrm{div}_v ( u^i_{[k]} ) \right\Vert _{L^\infty(\omega^i_{[k]})}
= \left\Vert \partial_v\left( \psi^i_{[k]}(x,v)\frac{v-{\bar v}^i}{|v-{\bar v}^i|}\right) \right\Vert _{L^\infty(\omega^i_{[k]})} \\
\leq  \frac{1}{\beta^i}+1 \leq \frac{1}{\bar\beta}+1  ,
$$
for every $t\in[0,\bar T)$. Then, applying the estimate \eqref{e-Lp} of Lemma \ref{prop_e-Lp}, we get $\Vert f(t)\Vert _{L^\infty}\leq \bar F$ with $\bar F= \Vert  f^0\Vert _{L^\infty} \exp((\phi(0)+1/\bar\beta+1)\bar T) <+\infty$. It follows that $\Vert f^i\Vert _\infty\leq \bar F$ for every $i$, which implies, by Lemma \ref{p-SS}, that $\varepsilon^i\geq \frac{c}{2\Vert f^i\Vert _\infty W^i}\geq \bar \varepsilon$ with $\bar \varepsilon=\frac{c}{2 \bar F W^0}>0$.
At this step, we have obtained that $\beta^i\geq \bar\beta$ and $\varepsilon^i\geq \bar\varepsilon$ for every $i$, and besides, we have $W^i\leq W^0$ for every $i$. Therefore $T^i=\min\left(\frac{\varepsilon^i}{W^i},\frac{\beta^i}{2c},1\right)\geq \min\left(\frac{\bar\varepsilon}{W^0},\frac{\bar\beta}{2c},1\right)$ does not converge to $0$, and hence $\bar T = \sum_{j=1}^\infty T^j=+\infty$. This contradicts the fact that $\bar T\leq nW^0$. We conclude that $\bar W=0$.

Since $W^i$ converges to $0$ as $i$ tends to $+\infty$, it follows that there exists $k\in\N^*$ such that $W^k<\eta$. This means that the iterative procedure terminates.

Recalling that $n=\lceil \frac{c}{2} \rceil$, the above arguments show that $\mu^k=\mu(\sum_{j=0}^k T^j)$ with $\sum_{j=0}^k T^j\leq W^0 \lceil \frac{2}c\rceil$, and we have $Y^k\leq Y^0+(W^0)^2 \lceil \frac{2}c\rceil$.

Finally, the constraints on the control follow from an iterative application of Lemma \ref{p-SS}.
\end{proof}

We now use this strategy to prove controllability to flocking, for any initial configuration $\mu^0\in\mathcal{P}^{ac}_c(\R\times\R)$.

\begin{theorem}[Flocking in 1D] 
Let $\mu^0\in\mathcal{P}^{ac}_c(\R\times\R)$ be such that $\supp(\mu^0)\subset[0,Y^0]\times [0,W^0]$. Let $c>0$ be arbitrary.
Then the complete strategy $\SSS$, applied with 
$$
\eta=\frac{1}{2}\int_{2\left(Y^0+\lceil\frac{c}2\rceil (W^0)^2\right)}^{+\infty} \phi(2x)\,dx ,
$$
provides a control satisfying the constraints \eqref{cont_u} and \eqref{cont_omega}, which steers the system from $\mu^0$ to the flocking region in time $T\leq W^0 \lceil \frac{2}{c}\rceil$. Then $\mu(t)$ converges to flocking.
\end{theorem}

\begin{proof}
Applying the strategy $\SSS$ with the given $\eta$ yields $\mu^k$ such that $W^k\leq \eta$. By Lemma \ref{p-SSeta}, we have $Y^k\leq Y^0+\lceil\frac{c}{2}\rceil (W^0)^2$, and hence
$$
2 W^k\leq 2\eta
=\int_{2(Y^0+\lceil\frac{c}2\rceil (W^0)^2)}^\infty \phi(2x)\,dx
\leq \int_{2 Y^k}^\infty \phi(2x)\,dx.
$$
Then, it follows from Corollary \ref{cor_flocksize} that $\mu(t)$ converges to flocking. The estimate on $T$ is given by Lemma \ref{p-SSeta}. 
\end{proof}

\subsection{Proof of Theorem \ref{mainthm} in dimension $d>1$}\label{s-SSd}
In dimension larger than one, we adapt the fundamental step $\SS$ and the complete strategy $\SSS$ of the one-dimensional case, as follows.

First of all, let us focus on a given coordinate. Let $j\in\{1,\ldots,d\}$ be arbitrary. Below, we describe the fundamental step $\SS_j$, adapted from the fundamental step $\SS$ in 1D.

\paragraph{Fundamental step $\SS_j$ for the $j^\textrm{th}$ component.} 
Let $\mu^0\in\mathcal{P}^{ac}_c(\R^d\times\R^d)$ be an initial datum. Using translations, we assume that $\supp(\mu^0)\subset \prod_{j=1}^d [0,Y^0_j]\times \prod_{j=1}^d [0,W^0_j]$, where $Y_j^0$ is the size of the support in the variable $x_j$ and $W_j^0$ is the size of the support in the variable $v_j$. As for the case $d=1$, admitting temporarily that the control that we will define produces a well defined absolutely continuous solution $\mu(t)$, we assume that 
$\supp(\mu(t))\subset \prod_{j=1}^d [0,Y_j(t)]\times \prod_{j=1}^d [a_j(t),a_j(t)+W_j(t)]$. We have $Y_j(0)=Y_j^0$, $a_j(0)=0$ and $W_j(0)=W_j^0$, for $j=1,\ldots,d$.
We define the functions
$$
\mathcal{Y}(t)=\sqrt{d} \prod_{j=1}^d  Y_j(t) ,\qquad
\mathcal{W}(t)= \sqrt{d} \prod_{j=1}^d (Y_j(t)+W_j(t))  -\mathcal{Y}(t) ,
$$
and we set $\mathcal{Y}^0=\mathcal{Y}(0)$ and $\mathcal{W}^0=\mathcal{W}(0)$.

We define the fundamental step $\SS_j$ similarly to $\SS$, with the following changes:
\begin{itemize}
\item $\alpha^+(t)=\frac{\phi(0)}{\phi(0)+\phi(\mathcal{Y}(t)+\mathcal{W}(t))} (W_j(t)+a_j(t)-{\bar v}_j(t))$, and similarly for $\beta^+(t)$, $\alpha^-(t)$ and $\beta^-(t)$;
\item the rectangle sets $\Omega_{[i]}$ are defined by
$$
\Omega_{[i]} = [0,Y_1^0]\times\cdots\times [0,Y_{j-1}^0]\times[x_{j,[i-1]},x_{j,[i]}]\times[0,Y_{j+1}^0]\times\cdots\times[0,Y_d^0]\times \prod_{j=1}^d [0,W_j^0],
$$ 
with the same mass requirements;
\item $\varepsilon^0$ is the largest positive real number such that
$$
\mu^0(\R\cdots\times [x_{j,[i-1]}-3\varepsilon^0,x_{j,[i]}+3\varepsilon^0]\times\cdots\times\R)\leq c, \quad\forall i\in\{1,\ldots,n\}.
$$
\item similarly, $\omega_{[i]}$ is defined with the interval $[x_{j,[i-1]}-2\varepsilon^0,x_{j,[i]}+2\varepsilon^0]$ on the $j$-th coordinate only;
\item the function $\psi_{[i]}$ is defined as in the 1D case, but depending on the coordinates $(x_j,v_j)$ only;
\item we define $u(t,x,v)=\psi(t,x,v)\frac{v_j-{\bar v}_j(0)}{|v_j-{\bar v}_j(0)|}$;
\end{itemize}

\begin{lemma}[fundamental step $\SS_j$ for the $j^\textrm{th}$ component.]\label{p-SSj}
The statement of Lemma \ref{p-SS} holds true for the fundamental step $\SS_j$, with the following changes:
\begin{itemize}
\item $Y_l^1\leq Y_l^0+W_l^0$, for $l=1,\ldots,n$;
\item the domain $\R^d \times \R^{j-1}\times [{\bar v}_j(0)-\alpha^0-\beta^0-k^-,{\bar v}_j(0)+\alpha^0+\beta^0+k^+]\times \R^{d-j-1}$ is invariant under the controlled particle flow $\Phi_{\omega,u}(t)$, for all $k^-\geq 0$ and $k^+\geq 0$;
moreover, all sets $\R^d \times \R^{l-1}\times [0,W_l^0]\times \R^{d-l-1}$, $l=1,\ldots, d$, are invariant as well;
\item either $[a_j^1,a_j^1+W^1_j]\subset [0,W^0_j-\frac{T^0}{n}]$ or $[a_j^1,a_j^1+W^1_j]\subset [\frac{T^0}{n},W^0_j]$, which implies that $W^1_j\leq W^0_j-\frac{T^0}{n}$;
\item $\displaystyle 12\varepsilon^0 \Vert f^0\Vert_{L^\infty(\R^d\times\R^d)} \prod_{\stackrel{k=1}{k\neq j}}^{d} Y_k^0 \prod_{k=1}^d  W_k^0\geq c$;
\end{itemize}
\end{lemma}

\begin{proof}
The proof is similar to the one of Lemma \ref{p-SS} and is skipped. Notice that $T^0\leq 1$ implies that the size of the support in the spatial variable $x_j$ increases from $Y_i^0$ at most to $Y_i^0+W_i^0$, which implies that the size of the spatial support is at most $\mathcal{Y}^0+\mathcal{W}^0$. The computation of $\alpha^+$, $\beta^+$, $\alpha^-$, $\beta^-$, gives all invariance properties.
\end{proof}

\paragraph{Complete strategy $\SSS_j$ for the $j^\textrm{th}$ component.}
Let $\eta>0$. We apply the fundamental step $\SS_j$ iteratively: while $W_j^i > \eta$, we compute $\mu^{i+1}=\mu(\sum_{l=0}^i T^l)$.

\smallskip

With arguments similar to the ones used to prove Lemma \ref{p-SSeta}, we establish that the above iteration terminates.

\begin{lemma}\label{p-SSjeta}
The statement of Lemma \ref{p-SSeta} holds true for the complete strategy $\SSS_j$, with the following changes:
\begin{itemize}
\item for every $j\in\{1,\ldots,d\}$, there exists $k_j\in\N^*$ such that the probability measure $\mu^{k_j}=f^{k_j}\,\dxv$, with support contained in $\left[0,Y_j^{k_j}\right]\times \left[a_j^{k_j},a_j^{k_j}+W_j^{k_j}\right]$, is such that $W_j^{k_j}\leq \eta$;
\item $\mu^{k_j}=\mu(\sum_{l=0}^{k_j} T^l)$ with $\sum_{l=0}^{k_j} T^l\leq W^0_j \lceil \frac{2}{c}\rceil$;
\item $Y^k_l\leq Y^0_l+W^0_l W^0_j \lceil \frac{2}{c}\rceil$, for every $l\in\{1,\ldots, d\}$.
\end{itemize}
\end{lemma}

\paragraph{Complete strategy $\SSS_*$.}
The complete strategy consists of applying successively the strategies $\SSS_j$, for $j=1,\ldots,d$.
In other words, by iteration on each component, we reduce the size of the velocity support in this component (with a bound $\eta$). In this process, the velocity support in the other components does not increase (but the spatial support may increase, according to Lemma \ref{p-SSj}).
At the end of these $d$ iterations, the velocity support is small enough (with a bound $\eta$) in all components. 
If $\eta$ is adequately chosen then this means that we have reached the flocking region.
Then, as in the 1D case, it follows from Corollary \ref{cor_flocksize} that $\mu(t)$ converges to flocking.

\begin{theorem}[Flocking in multi-D] 
Let $\mu^0\in\mathcal{P}^{ac}_c(\R^{2})$ be such that $\supp(\mu^0)\subset \prod_{j=1}^d [0,Y_j^0]\times \prod_{j=1}^d [0,W_j^0]$. Let $c>0$ be arbitrary.
We set
$$
W_*=\lceil \frac{2}c\rceil\sum_{j=1}^d W_j^0,\qquad 
\tilde{W}=\sqrt{d} \prod_{j=1}^d (Y_j^0+W_j^0W_*).
$$
Then the strategy $\SSS$, applied with 
$$
\eta=\frac{1}{2\sqrt{d}}\int_{\tilde{W}}^\infty \phi(2x)\,dx,
$$ 
provides a Lipschitz control satisfying the constraints $\mu(t)(\omega(t))\leq c$ and $\Vert u(t)\Vert_{L^\infty(\R^d\times\R^d)}\leq 1$, which steers the system \eqref{kinetic_CS} from $\mu^0$ to the flocking region in time less than or equal to $W_*$. Then $\mu(t)$ converges to flocking.
\end{theorem}

\begin{proof}
Let us consider the $j^\textrm{th}$ step of the strategy, along which we apply $\SSS_j$, and at the end of which we have obtained $\mu^j$. By construction, the velocity size in the $j$-th component is less or equal than $\eta$, while the velocity size in the other components does not increase, as a consequence of Lemma \ref{p-SSj}. 

Note that, using Lemma \ref{p-SSjeta}, the duration of this $j^\textrm{th}$ step is less than or equal to $W_j^0 \lceil \frac{2}c\rceil$. Hence, the total time of the procedure is less than or equal to $\lceil \frac{2}{c}\rceil\sum_{l=1}^d W_l^0=W_*$. 

Let us now investigate the evolution of the size of the velocity support in the variable $v_j$ along the whole procedure. After having applied the strategies $\SSS_1,\ldots,\SSS_{j-1}$, the size of the velocity support in the variable $v_j$ is less than or equal to $W_j^0$; the application of the stragegy $\SSS_j$ decreases this size at some value less than or equal to $\eta$; then, the  application of the strategies $\SSS_{j+1},\ldots,\SSS_d$ keeps this size at some value less than or equal to $\eta$. As a result, the size of the velocity supports of each component is less than or equal to $\eta$ at the end of the procedure. Finally, the velocity support of $\mu^d$ is contained in the ball $B(\tilde v,\frac{\eta\sqrt{d}}{2})$, with $\tilde v=(\tilde a_1,\tilde a_2,\ldots,\tilde a_d)+\frac{\eta}{2}(1,\ldots,1)$, where $\tilde a_i=\min(v_i \in \R\ |\ (x,v)\in \supp(\mu^d))$.

Let us now investigate the evolution of the size of the spatial support. Consider the evolution of the size of the space support in the variable $x_j$ for the whole algorithm. Since the size of the velocity support in the variable $v_j$ is always bounded by $W_j^0$, it follows that $Y_j$ may increase of at most $W_j^0 W_*$. Then the space support of $\mu^d$ is contained in the ball $B(\tilde x,\frac{\tilde{W}}{2})$, with $\tilde x=(\tilde x_1,\ldots,\tilde x_d)$ and $\tilde x_j= \frac{Y_j^0+W_j^0 W_*}{2}$.

Now, to conclude that $\mu(t)$ converges to flocking, it suffices to apply Corollary \ref{cor_flocksize}, since 
$2\frac{\sqrt{d}\eta}2=\frac12 \int_{2\frac{\tilde{W}}2}^\infty \phi(2x)\,dx<\int_{\tilde{W}}^\infty \phi(2x)\,dx$.
\end{proof}

\subsection{Proof of the variant of Theorem \ref{mainthm}}\label{s-space}
In this section, we consider the controlled kinetic Cucker-Smale equation \eqref{kinetic_CS} with the control constraints \eqref{cont_u} and \eqref{cont_omega_variante}.
We restrict our study to the one-dimensional case, the generalization to any dimension being similar to that done in Section \ref{s-SSd}.

We first define the fundamental step of our strategy. Here, the objective is to make decrease the velocity support from $[0,W^0]$ to $[0,\eta]$. We only act on the upper part of the interval. For this reason, we need to define $\alpha^0$, $\beta^0$ only (and not $\alpha^+$, $\beta^+$, $\alpha^-$, $\beta^-$, as in the problem of control with constraint on the crowd). We also can assume $a=0$ for all times.

\paragraph{Fundamental step $\TT$.}
Let $\mu^0\in\mathcal{P}^{ac}_c(\R\times\R)$ be such that $\supp(\mu^0)\subset [0,Y^0]\times [0,W^0]$. 
Let ${\bar v}^0\in(0,W^0)$ be the velocity barycenter of $\mu^0$. 
Using notations similar to those used in Section \ref{s-ss}, we define the functions
$$
\alpha(t) = \frac{\phi(0)}{\phi(0)+\phi(Y(t)+W(t))}(W(t)-{\bar v}(t)), \qquad
\beta(t) =\frac{1}{3}\frac{\phi(Y(t)+W(t))}{\phi(0)+\phi(Y(t)+W(t))}(W(t)-{\bar v}(t)) ,
$$
and we set $\alpha^0=\alpha(0)$, $\beta^0=\beta(0)$, and
$$
\varepsilon^0 =  \min\left( \frac{1}{2}\beta^0,\frac{\sqrt{(Y^0)^2+2c(W^0+1)}-Y^0}{2(W^0+2)} \right) .
$$
We define the (positive) time $T^0=\varepsilon^0$. The fact that $\varepsilon^0$ represents both a distance and a time is due to the fact that the velocity constraint on $u$ is equal to $1$.

Along the time interval $[0,T^0]$, we define the constant control set $\omega(t)=\omega^0$, with $\omega^0=[-\varepsilon^0,Y^0+\varepsilon^0 W^0+\varepsilon^0]\times [W^0-2\varepsilon^0,W^0+2\varepsilon^0]$, and we define the (constant in time) control function $u(t,x,v)=u^0(x,v)$, with $u^0(x,v)=\psi(x)\zeta(y)$, and
\begin{equation*}
\psi(x)= \left\{ \begin{array}{lll}
0& \textrm{if} & x< -\varepsilon^0 , \\
\frac{x+\varepsilon^0}{\varepsilon^0}& \textrm{if} & x\in[-\varepsilon^0,0) , \\
1& \textrm{if} & x\in [0,Y^0+\varepsilon^0 W^0) , \\
\frac{-x+Y^0+\varepsilon^0 W^0+\varepsilon^0}{\varepsilon^0}& \textrm{if} & x\in[Y^0+\varepsilon^0 W^0,Y^0+\varepsilon^0 W^0+\varepsilon^0) , \\
0 & \textrm{if} & x\geq Y^0+\varepsilon^0 W^0+\varepsilon^0 ,
\end{array} \right.
\end{equation*}
and
\begin{equation*}
\zeta(v)= \left\{ \begin{array}{lll}
0& \textrm{if} & v< W^0-2\varepsilon^0 , \\
\frac{W^0-2\varepsilon^0-v}{\varepsilon^0}& \textrm{if} & v\in[W^0-2\varepsilon^0, W^0-\varepsilon^0) , \\
-1& \textrm{if} & v\in [W^0-\varepsilon^0,W^0+\varepsilon^0) , \\
\frac{v-(W^0+2\varepsilon^0)}{\varepsilon^0}& \textrm{if} & v\in[W^0+\varepsilon^0, W^0+2\varepsilon^0) , \\
0 & \textrm{if} & v\geq W^0+2\varepsilon^0 .
\end{array} \right.
\end{equation*}

\medskip

The next result states that the fundamental step $\TT$ is well defined, and that this control strategy makes  the velocity support of the crowd decrease. 

\begin{lemma}\label{p-TT} 
Let $\mu^0=f^0\, \dxv\in\mathcal{P}^{ac}_c(\R\times\R)$, with compact support contained in $[0,Y^0]\times [0,W^0]$.
There exists a unique solution $\mu\in C^0([0,T^0],\mathcal{P}(\R\times\R))$ of \eqref{kinetic_CS},  corresponding to the control $\chi_\omega u$ defined by $\TT$. Moreover:
\begin{itemize}
\item $\mu\in C^0([0,T^0],\mathcal{P}^{ac}_c(\R\times\R))$, that is, the solution $\mu$ remains, like $\mu^0$, absolutely continuous and of compact support;
in particular, at time $T^0$, we have $\mu^1=\mu(T^0)\in\mathcal{P}^{ac}_c(\R\times\R)$;
\item the sets $\R\times[0,W^0]$ and $\R\times [0,W^0-\varepsilon^0]$ are invariant under the controlled particle flow $\Phi_{\omega,u}(t)$ (defined in Corollary \ref{cor_appli_CS_control});
\item setting $Y^1=Y(T^0)$ we have $Y^1\leq Y^0+\varepsilon^0 W^0$ and $0\leq W^1\leq W^0-\varepsilon^0$;
\item the control satisfies the constraints \eqref{cont_u} and \eqref{cont_omega_variante}.
\end{itemize}
\end{lemma}

\begin{proof}
The proof of the fact that $\mu\in C^0([0,T^0],\mathcal{P}^{ac}_c(\R\times\R))$ is similar to the proof of Lemma \ref{p-SS}.

The set $\R\times[0,W^0]$ is invariant under the controlled particle flow $\Phi_{\omega,u}(t)$, because by construction the vector field $\xi[\mu(t)]$ and $u^0$ point inwards along the boundary of that domain. Since $\supp(\mu^0)\subset [0,Y^0]\times[0,W^0]$, it follows that $\supp(\mu(t))\subset [0,Y^0+\varepsilon^0 W^0]\times[0,W^0]$ for $t\in[0,T^0]$ because $T^0=\varepsilon^0$. In particular we get that $Y^1\leq Y^0+\varepsilon^0 W^0$.

The proof of the fact that the set $\R\times [0,W^0-\varepsilon^0]$ is invariant under the controlled particle flow $\Phi_{\omega,u}(t)$ is similar to the proof of Lemma \ref{p-SS}, noting that the velocity barycenter ${\bar v}(t)$ satisfies $\vert{\bar v}(t)-{\bar v}(0)\vert<\beta^0$ and thus that the vector field $\xi[\mu(t)]$ points inwards at any point $(x,v)$ such that $v\geq {\bar v}(0)+\alpha^0+\beta^0$.

Recall that $[0,W(t)]$ is the velocity support of $\mu(t)$. Since the set $\R\times [0,W^0]$ is invariant, we have $W(t)\leq W^0$ for every $t\in [0,T^0]$. Let us prove that $W^1=W(T^0)\leq W^0-\varepsilon^0$. By contradiction, let us assume that $W^1>W^0-\varepsilon^0$. Then $W(t)>W^0-\varepsilon^0$ for every $t\in[0,T^0]$, otherwise there would exist $\bar t\in[0,T^0]$ such that $W(\bar t)\leq W^0-\varepsilon^0$, and then $W^1=W(T^0)\leq W^0-\varepsilon^0$ by invariance of the set $\R\times[0,W^0-\varepsilon^0]$ under the controlled particle flow. Since $\beta^0>\varepsilon^0$, it follows that $W(t)\geq W^0-\beta^0$ on the whole interval $[0,T^0]$, and then the velocity component of the vector field acting on any $(x,v)$ with $v=W(t)$ is $\xi[\mu(t)](x,v)+u(x,v)$. But one has $\xi[\mu(t)](x,v)<0$ because $v-{\bar v}(t)>\alpha^0$, and $u(x,v)=-1$ by definition of $u$. Since this estimate holds for any $(x,v)\in\omega^0(t)$ with $v=W(t)$, it follows that $\dot W(t)<- 1$. Since this holds true for every $t\in [0,T^0]$, we infer that $W(T)\leq W^0-\varepsilon^0$, which is a contradiction.

Finally, let us prove that the control satisfies the constraints.
The control $\chi_{\omega} u$ satisfies ${\bf (H)}$ and $\Vert u(t)\Vert _{L^\infty(\R^d\times\R^d)}\leq 1$ by construction.
The constraint $\vert\omega(t)\vert\leq c$ follows from the choice of $\varepsilon^0$. Indeed, by construction we have $\vert\omega(t)\vert\leq 4\varepsilon^0 (Y^0+\varepsilon^0 W^0+2\varepsilon^0)$, and solving the equation $4\varepsilon (Y^0+\varepsilon W^0+2\varepsilon)=c$ yields $\varepsilon= (\sqrt{(Y^0)^2+c(W_0+2)}-Y_0)/2(W_0+2)$. But we have chosen $\varepsilon^0$ such that $\varepsilon^0\leq (\sqrt{(Y^0)^2+2cW^0+2c}-Y^0) / 2(W^0+2)$.
\end{proof}

As previously, the complete strategy consists of applying iteratively the fundamental step $\TT$ until the size of the velocity support decreases under a threshold $\eta$.

\paragraph{Complete strategy $\TTT$.}
Let $\eta>0$. We apply the fundamental step $\TT$ iteratively: while $W^i> \eta$, we compute $\mu^{i+1}=\mu(\sum_{j=0}^i T^j)$.

\medskip

As before, we establish that the above iteration terminates.

\begin{lemma}\label{p-TTeta} 
Let $\mu^0\in\mathcal{P}^{ac}_c(\R\times\R)$ be such that $\supp(\mu^0)\subset [0,Y^0]\times [0,W^0]$, and let $\eta>0$. Then there exists $k\in\N^*$ such that the probability measure $\mu^k=f^k\, \dxv$, with support contained in $[0,Y^k]\times [0,W^k]$, is such that $W^k\leq \eta$. Moreover, we have $\mu^k=\mu(\sum_{j=0}^k T^j)$ with $\sum_{j=0}^k T^j\leq W_0$, and we have $Y^k\leq Y^0+(W^0)^2$. Furthermore, the control satisfies the constraints \eqref{cont_u} and \eqref{cont_omega_variante}.
\end{lemma}

\begin{proof}
Consider the sequence of positive real numbers $\varepsilon^i$ obtained by the iterative application of the fundamental step $\TT$. According to Lemma \ref{p-TT}, we have $W^{i+1}\leq W^i-\varepsilon^i$ for every $i$, and since $W^i\geq 0$, it follows that $\sum_{j=1}^i \varepsilon^i\leq W_0$ for every $i$. Setting $\bar{T}=\sum_{j=1}^{+\infty} T^j$, we have $\bar T\leq W^0$. As a consequence, the controlled particle flow $\Phi_{\omega,u}(t)$ lets the set $[0,Y^0+(W^0)^2]\times [0,W^0]$ invariant, for every time $t\in[0,\bar{T})$, where the time interval is open at $\bar T$ since we have not proved yet the convergence of the complete procedure. Note that this implies that $Y^i\leq Y^0+(W^0)^2$ for every $i$.
Since the sequence $(W^i)_{i\in\N^*}$ is bounded below by $0$, bounded above by $W^0$, and is decreasing (because $W^{i+1}\leq W^i-\varepsilon^i$ with $\varepsilon^i>0$), it converges to some limit $\bar W\geq 0$. Let us prove that $\bar W=0$. By contradiction, let us assume that $\bar W>0$. Then, for any given $i$, we have $W^i\geq \bar W$ and either $W^i-{\bar v}^i\geq \frac{W^i}{2}\geq \frac{\bar{W}}{2}$ or ${\bar v}^i\geq \frac{W_i}{2}\geq \frac{\bar{W}}{2}$. In both cases, we have 
$$
\beta^i\geq\frac{1}{3} \frac{\phi(Y^i+W^i)}{\phi(0)+\phi(Y^i+W^i)} \frac{\bar{W}}{2}\geq \frac{\phi(Y^0+(W^0)^2+W^0)}{\phi(0)+\phi(\bar W)} \frac{\bar{W}}{2},
$$
where we have used that $0\leq Y^i\leq (Y^0+W^0)^2$, that $\bar W\leq W^i\leq W^0$ and that $\phi$ is decreasing. Since the estimate does not depend on $i$, it follows that $\beta^i\geq \bar\beta$ for every $i$, with $\bar \beta=\frac{\phi(Y^0+(W^0)^2+W^0)}{\phi(0)+\phi(\bar W)} \frac{\bar{W}}{2}>0$.
Similarly, note that $\bar W\leq W^i\leq W^0$ implies 
$$
\frac{\sqrt{(Y^i)^2+2cW^i+2c}-Y^i}{2(W^i+2)}\geq \frac{\sqrt{(Y^i)^2+2c\bar W+2c}-Y^i}{2(W^0+2)} = h(Y^i).
$$
The function $h$ is decreasing with respect to $Y^i$ in the interval $Y^i\in[0,Y^0+(W^0)^2]$, and reaches its minimum for $Y^i=Y^0+(W^0)^2$, therefore
$$
\frac{\sqrt{(Y^i)^2+2cW^i+2c}-Y^i}{2(W^i+2)}\geq \bar \gamma=
\frac{\sqrt{(Y^0+(W^0)^2)^2+2c\bar W+2c}-(Y^0+(W^0)^2)}{2(W^0+2)}>0.
$$
It follows that $\varepsilon^i\geq  \min\left(\frac{1}{2}\bar\beta,\bar \gamma\right)$, and since $\bar\beta$ and $\bar \gamma$ do not depend on $i$, $\varepsilon^i$ does not converge to $0$. This contradicts the fact that $\sum_{j=1}^\infty T^j=\sum_{j=1}^\infty \varepsilon^j\leq W^0$.
Therefore, $W^i$ converges to $0$ as $i$ tends to $+\infty$, and it follows that there exists $k$ such that $W^k<\eta$, which means that the algorithm terminates.

For $i=k$, we have obtained $\mu^k=\mu(\sum_{j=0}^k T^j)$ with $\sum_{j=0}^k T^j\leq W_0$, and $Y^k\leq Y^0+(W^0)^2$.

To prove that the constraints on the control are satisfied, it suffices to apply Lemma \ref{p-TT} for the $k$ steps.
\end{proof}

Now, it suffices to choose adequately $\eta$ to obtain flocking.

\begin{theorem}[Flocking in 1D] 
Let $\mu^0\in\mathcal{P}^{ac}_c(\R\times\R)$ be such that $\supp(\mu^0)\subset [0,Y^0]\times [0,W^0]$, and let $c>0$ be arbitrary. Then, the strategy $\TTT$ applied with 
$$
\eta=\frac{1}{2}\int_{2(Y^0+ (W^0)^2)}^\infty \phi(2x)\,dx
$$ 
provides a control satisfying the constraints \eqref{cont_u} and \eqref{cont_omega_variante}, which steers the system \eqref{kinetic_CS} to the flocking region in time less than or equal to $W^0$. Then $\mu(t)$ converges to flocking.
\end{theorem}

\begin{proof}
By Lemma \ref{p-TTeta}, we have $Y^k\leq Y^0+ (W^0)^2$ and 
$$
2 W^k\leq 2\eta=\int_{2(Y^0+ (W^0)^2)}^\infty \phi(2x)\,dx\leq \int_{2 Y^k}^\infty \phi(2x)\,dx.
$$
Using Corollary \ref{cor_flocksize}, the flocking property follows. The estimate on the time at which $\mu(t)$ has reached the flocking region follows from Lemma \ref{p-TTeta}, and the conditions on the control follow from Lemma \ref{p-TTeta}.
\end{proof}

\noindent{\bf Acknowledgment.}
This work was initiated during a visit of F. Rossi and E. Tr\'elat to Rutgers University, Camden, NJ, USA. They thank the institution for its hospitality.\\
The work was partially supported by the NSF Grant \#1107444 (KI-Net: Kinetic description of emerging challenges in multiscale problems of natural sciences), and by the Grant FA9550-14-1-0214 of the EOARD-AFOSR.

\end{document}

%% file: fig_supp.pdf_t
\begin{picture}(0,0)%
\includegraphics{fig_supp.pdf}%
\end{picture}%
\setlength{\unitlength}{3947sp}%
\begingroup\makeatletter\ifx\SetFigFont\undefined%
\gdef\SetFigFont#1#2#3#4#5{%
  \reset@font\fontsize{#1}{#2pt}%
  \fontfamily{#3}\fontseries{#4}\fontshape{#5}%
  \selectfont}%
\fi\endgroup%
\begin{picture}(8424,8124)(2914,-7798)
\put(4951,-5086){\makebox(0,0)[lb]{\smash{{\SetFigFont{29}{34.8}{\rmdefault}{\mddefault}{\updefault}{\color[rgb]{0,0,0}$\supp(\mu)$}%
}}}}
\put(8776,-436){\makebox(0,0)[lb]{\smash{{\SetFigFont{29}{34.8}{\rmdefault}{\mddefault}{\updefault}{\color[rgb]{0,0,0}$(x,v)$}%
}}}}
\put(3526,-61){\makebox(0,0)[rb]{\smash{{\SetFigFont{29}{34.8}{\rmdefault}{\mddefault}{\updefault}{\color[rgb]{0,0,0}$v$}%
}}}}
\put(10876,-7261){\makebox(0,0)[lb]{\smash{{\SetFigFont{29}{34.8}{\rmdefault}{\mddefault}{\updefault}{\color[rgb]{0,0,0}$x$}%
}}}}
\end{picture}%